\documentclass[draft, 9pt]{amsart}
\numberwithin{equation}{subsection}

\usepackage{amsfonts}
\usepackage{latexsym, amssymb}
\newcommand{\field}[1]{\mathbb{#1}}
\newcommand{\R}{\field{R}}
 \newcommand{\Sp}{\field{S}}

\newcommand{\C}{\field{C}}
\newcommand{\N}{\field{N}}
\def\H{\mathsf{H}}

\def\D\theta ij#1{\dis \frac{\partial #1}{\partial \theta_i^j}}

\def\sqr{{\hskip1pt\vcenter{\vbox{\hrule height.4pt
\hbox{\vrule width.4pt height4pt\kern4pt
\vrule width.4pt}
\hrule height.4pt}}}}

\def\im{\textrm{Im}}
\def\cal{\mathcal}

\def\trans{\ifmmode{\rm \frown\mkern-16.8mu \vert
\mkern8mu}\else{$\frown\mkern-16.8mu\vert\mkern8mu$}\fi\relax}
\def\ap{\rightarrow}

\def\dis{\displaystyle}

\def\D{\Delta}
\def\dis{\displaystyle}
\def\a{\alpha}
\def\b{\beta}
\def\g{\gamma}
\def\G{\Gamma}
\def\t{\tau} \def\d{\delta}

\def\r{\rho}
\def\l{\lambda}

\def\s{\sigma}
\def\S{\Sigma}

\def\O{\Omega}

\def\p{\partial}

\def\o{\omega}
\newtheorem{The1}{Theorem}
\newtheorem{The}{Theorem}[subsection]
 \newtheorem{Pro}{Proposition}[subsection]
\newtheorem{Def}{Definition}[subsection]
\newtheorem{Cor}{Corollary}[subsection]
\newtheorem{Lem}{Lemma}[subsection]

\newtheorem{Rem}{Remark}[subsection]

\newtheorem{Ass}{Assumptions}[subsection]
\newtheorem{Obs}{Observation}[subsection]

\title{{ M\"{o}bius transformations and the configuration space of a Hilbert snake}}
\begin{document}
\maketitle

  \begin{center}\author{{F.  PELLETIER, }\footnote{Universit\'e de Savoie, Laboratoire de Math\'ematiques (LAMA)
Campus Scientifique,
73376 Le Bourget-du-Lac Cedex, France.
E-Mail: pelletier@univ-savoie.fr.}
 {R. SAFFIDINE }\footnote{Universit\'e Ferhat Abbas, S\'etif-$1$, facult\'e des sciences, d\'epartement de math\'ematiques, Alg\'erie: r\_ saffidine@yahoo.fr }
  {\& N. BENSALEM}\footnote{Universit\'e Ferhat Abbas, S\'etif-$1$, facult\'e des sciences, d\'epartement de math\'ematiques, Alg\'erie:  naceurdine\_bensalem@yahoo.fr }}
   \end{center}

\begin{abstract}
The purpose of this paper is to give a simpler proof to the problem of controllability of a Hilbert snake \cite{PeSa}.
Using the action of the M\"{o}bius group of the unite sphere on the configuration space, in the context of a separable Hilbert  space.
We give a generalization of the  Theorem of accessibility contained in \cite{Ha} and \cite{Ro} for articulated arms and snakes  in a finite dimensional Hilbert space.
    \end{abstract}

\smallskip

\noindent\textit{classification}: 22F50, 34C40, 34H, 53C17, 53B30, 53C50, 58B25, 93B03.


\date{}
%%%%%%%%%%%%%%%%%%%%%%%%%%%%%%%%%%%%%%%%%%%%%%%%%%%%%%%%%%%%%%%%%%%%%%%%%%%
\section{Introduction and results}\label{intro}
%%%%%%%%%%%%%%%%%%%%%%%%%%%%%%%%%%%%%%%%%%%%%%%%%%%%%%%%%%%%%%%%%%%%%%
%%%%%%%%%%%%%%%%%%%%%%%%%%%%%%%%%%%%%%%%%%%%%%%%%%%%%%%%%%%%%%%%%%%%%%
 The group of M\"{o}bius transformations  of a finite dimensional space  is generated by inversions of spheres. It is one of the fundamental geometrical groups. M\"{o}bius transformations
preserve spherical shapes and also
 the angles between pairs of curves. This group can be considered as the conformal group of the  sphere identified with  the  compactification of a finite dimensional space.

  If we denote by $\mathfrak{M}(\Sp^n)$  the M\"{o}bius transformations of such a sphere $\Sp^n$ which preserve the orientation,  it is known that $\mathfrak{M}(\Sp^n)$  is isomorphic to the group $SO_0(n,1)$   which is the connected component of the identity of $O(n,1)$.  All these results can be generalized to the context of a Hilbert space (cf. \cite{Be1} and  \cite{La} for instance).  Therefore the group $\mathfrak{M}(\Sp_{\mathbb{H}})$ of M\"{o}bius transformations of the unit sphere $\Sp_{\mathbb{H}}$ of a Hilbert space $\mathbb{H}$ is also isomorphic to some subgroup $SO_{0}(\mathbb{H},1)$ of the group $O(\mathbb{H},1)$ of linear Lorentz transformations of  a Lorentz structure  on  ${\cal H}=\R\oplus\mathbb{H}$ (for more details see Subsection \ref{Alorentz}). If we consider $SO(\mathbb{H},1)$ as a subgroup of the group $GL({\cal H})$ of continuous automorphisms of $\cal H$, we can look for the intersection $SO_{HS}(\mathbb{H},1)$ of $SO(\mathbb{H},1)$ with the subgroup $GL_{{HS}}({\cal H})$ of Hilbert-Schmidt automorphisms of $GL({\cal H})$. According to \cite{Go}, $SO_{HS}(\mathbb{H},1)$ can be seen as  a limit of an  increasing sequence
  \begin{eqnarray}\label{suitecrois}
SO(\mathbb{H}_{2},1)\subset\cdots\subset SO(\mathbb{H}_{n},1)\subset \cdots \subset SO_{HS}(\mathbb{H},1)\subset  GL_{{HS}}({\cal H}).\nonumber
\end{eqnarray}
 Via the previous isomorphism from $SO (\mathbb{H},1)$ to  $\mathfrak{M}(\Sp_{\mathbb{H}})$ we obtain a subgroup $\mathfrak{M}_{HS}(\Sp_{ \mathbb{H}})$ of the group of M\"{o}bius transformations of the unit sphere $\Sp_{\mathbb{H}}$.

On the other hand, as in the finite dimensional situation, the Lie algebra $\mathfrak{g}$ of $SO_{HS}(\mathbb{H},1)$ has a decomposition of type
$\mathfrak{g}=\mathfrak{h}\oplus \mathfrak{s}$
where $\mathfrak{s}$ is the Lie algebra of the subgroup $SO_{{HS}}({\cal H})$ of the Hilbert-Schmidt isometries of the Hilbert space $\cal H$. Again $\mathfrak{h}$ can be obtain as an adequate limit of finite dimensional subspace $\mathfrak{h}_{n}$ which is a factor of the classical  decomposition $\mathfrak{g}_{n}=\mathfrak{h}_{n}\oplus\mathfrak{s}_{n} $ of the Lie algebra $\mathfrak{g}_{n}$ of  $SO(\mathbb{H}_{n},1)$. In this way  we get a natural sub-Riemannain structure on $\mathfrak{M}_{{HS}}(\Sp_{\mathbb{H}})$ at the same time directly from $\mathfrak{h}$ and as limit of the canonical sub-Riemannian structure on each $SO(\mathbb{H}_{n},1)$. Now,  we know that in the finite dimensional case,  each pair of  elements of $SO(\mathbb{H}_{n},1)$ can be joined by a horizontal path. Unfortunately, this is no longer true in $\mathfrak{M}_{{HS}}(\Sp_{\mathbb{H}})$.
Our first result is to proved that there exists in $\mathfrak{M}_{{HS}}(\Sp_{\mathbb{H}})$ a dense subgroup  $\mathfrak{M}^{1}_{{HS}}(\Sp_{\mathbb{H}})$, provided with its own Lie Banach group structure, such that each pair of  elements of $\mathfrak{M}^{1}_{{HS}}(\Sp_{\mathbb{H}})$ can be joined by a horizontal path (cf. Theorem \ref{subR}). This Theorem allows us to give a simpler proof of  the accessibility  result in the problem of a Hilbert snake obtained in \cite{PeSa}.\\

More precisely, recall that a Hilbert snake of length $L$ is a continuous piecewise
$C^1$-curve $S : [0,L] \ap  \field{H}$, arc-length parameterized  such that  $S(0) = 0$.   Given a fixed partition $\cal P$ of $[0,L]$, the set  ${\cal C}^L_{\cal P}$ of such curves will be called the configuration set and carries a natural structure of Banach manifold. To any "configuration" $u\in {\cal C}^L_{\cal P}$ is naturally associated the end map:
${\cal E}(u)=\dis\int_0^Lu(s)ds$. This map is smooth and its kernel has a canonical complemented subspace which gives rise to a closed distribution $\cal D$ on ${\cal C}^L_{\cal P}$.
The problem of controllability of the ''head'' $S(L)$ of a Hilbert snake can be transformed in the following accessibility problem in ${\cal C}^L_{\cal P}$ (cf. section \ref{result}):

Given an initial (resp. final) configuration $u_0$ (resp. $u_1$) in ${\cal C}^L_{\cal P}$, such that ${\cal E}(u_i)=x_i$, $i=0,1$, find a piecewise $C^1$ horizontal curve $\g:[0,T]\ap {\cal C}^L_{\cal P}$ (i.e. $\g$ is tangent to $\cal D$) and  which joins $u_0$ to $u_1$. \\

Therefore, given any configuration $u\in {\cal C}^L_{\cal P}$ we have to look for the accessibility set ${\cal A}(u)$ of  all configurations $v\in {\cal C}^L_{\cal P}$ which can be joined from $u$ by a piecewise $C^1$ horizontal curve. It is shown in \cite{PeSa} that there exists a  canonical distribution $\bar{\cal D}$ which contains the previous horizontal distribution $\cal D$ which is integrable and each accessibility set ${\cal A}(u)$ is a dense subset of the maximal integral manifold of $\bar{\cal D}$  which contains $u$.

As in the finite dimensional case (see \cite{Ha} and  \cite{Ro}),  we have  a natural action $\mathfrak{A}$ of the group $\mathfrak{M}(\Sp_{\mathbb{H}})$
 on ${\cal C}^L_{\cal P}$.  Since  we have a canonical isomorphism between $\mathfrak{M}_{{HS}}(\Sp_{\mathbb{H}})$ and  $SO_{HS}({\mathbb{H}},1)$,  let  $\mathfrak{M}^{1}_{{HS}}(\Sp_{\mathbb{H}})$ be  the subgroup of $\mathfrak{M}_{{HS}}(\Sp_{\mathbb{H}})$ which is associated to $SO^{1}_{HS}({\mathbb{ H}},1)\subset SO_{HS}({\mathbb{ H}},1)$. Then we have the follwing result

 \begin{The1}\label{1}${}$
 \begin{enumerate}
\item The orbit  through $u\in{\cal C}^L_{\cal P}$ of the restriction of the action $\mathfrak{A}$ to  $\mathfrak{M}_{{HS}}(\Sp_{\mathbb{H}})$  is exactly the maximal integral manifold ${\cal L}(u)$ of $\bar{\cal D}$ which contains $u$.
\item The orbit  ${\cal A}^{1}(u)$ through $u\in{\cal C}^L_{\cal P}$ of the restriction of the action $\mathfrak{A}$ to  $\mathfrak{M}^{1}_{{HS}}(\Sp_{\mathbb{H}})$ is contained in ${\cal A}(u)$ and it is a dense subset of ${\cal L}(u)$. In particular ${\cal A}(u)$ is a dense subset of ${\cal L}(u)$.
 \end{enumerate}
 \end{The1}

This paper is organized as follows.  Section \ref{Mobius}  contains in its first part  all the definitions and results about M\"{o}bius transformations in the  Hilbert space context  which are needed to prove the announced results of the sub-Riemannian structure on $\mathfrak{M}_{{HS}}(\Sp_{\mathbb{H}})$. Properties of the Hilbert-Schmidt group $\mathfrak{M}_{{HS}}(\Sp_{\mathbb{ H}})$ are described in Section \ref{SubRM}.   In Section \ref{HilbSn}, according to \cite{PeSa} we first  recall all the context concerning the problem of controllability of a Hilbert snake. Then we apply the results of Section \ref{Mobius}  to prove Theorem \ref{1}. Finally some technical proofs used in Section \ref{Mobius}  are presented in Section \ref{A}.

%%%%%%%%%%%%%%%%%%%%%%%%%%%%%%%%%%%%%%%%%%%%%%%%%%%%%%%%%%%%%%%%%%%%%%%%%%%%%%
%%%%%%%%%%%%%%%%%%%%%%%%%%%%%%%%%%%%%%%%%%%%%%%%%%%%%%%%%%%%%%%%%%%%%%%%%%%%
\section{M\"{o}bius transformations of a Hilbert space}\label{Mobius}
%%%%%%%%%%%%%%%%%%%%%%%%%%%%%%%%%%%%%%%%%%%%%%%%%%%%%%%%%%%%%%%%%%%%%%%%%%%%%%%%%%
\subsection{M\"{o}bius transformations}${}$\\
%%%%%%%%%%%%%%%%%%%%%%%%%%%%%%%%%%%%%%%%%%%%%%%%%%%%%%%%%%%%%%%%%%%%%%%%%%%%%%%%%\texttt{}
%%%%%%%%%%%%%%%%%%%%%%%%%%%%%%%%%%%%%%%%%%%%%%%%%%%%%%%%%%%%%%%%%%%%%%%%
In this paper $\mathbb{H}$ is a fixed Hilbert space on $\R$ and $\{ e_{i}\}_{i\in I}$ will denote a Hilbert basis of  $\mathbb{H}$ where $I$ is either the finite set $\{1,\cdots,n\}$ with $ n\geq 2$ or $I=\N\setminus\{0\}$ and we
denote by  $<\;,\;>$ the inner product  and $|\;|
$ the associated norm. With these notations, we can identify $\mathbb{H}$ with $l^2(I)$ and each $x\in\mathbb{H}$ is identified with the sequence $(x_i)$ where $x_i=< x,e_i>, \;:i\in I$.\\

Let  $\H$  be any  hyperplane   in $\mathbb{H}$. We can always choose a Hilbert basis  $\{ e_{i}\}_{i\in I}$ of $\mathbb{H}$ such that  $\{e_i\}_{i>1}$ is a Hilbert basis of $\mathsf{H}$  and we
also denote by  $<\;,\;>$ the induced  inner product  in $\H$ and $|\;|$ the associated norm. With these notations, we consider the set  ${\widehat{\H}}=\mathsf{H}\cup\Big\{\infty\Big\}$ equipped with the following topology:  \ $U\subset{\widehat{\H}}$ is an open set if and only if $U\cap\mathsf{H}$ is an open set  and $\mathsf{H} \backslash U$ is a bounded set in $\H$, if  $\infty\in U$.
In this section we will recall the classical properties of the M\"{o}bius  transformations of $\H$  which the reader can find in ( \cite{Be1}, \cite{Be} and \cite{La} ). We first introduce the following notations:

$\bullet$ Given $a\in\mathsf{H}$ and $r,t\in\mathbb{R}$ with  $r>0$,  a M\"{o}bius sphere in  ${\widehat{\H}}$  is either a classical sphere in $\H$:
\begin{equation}
    S(a,r)=\Big\{x\in\mathsf{H}\;:  |x-a|=r\Big \},
\end{equation}
 or an extended hyperplane :	
\begin{equation}
 P(a,t)=\Big\{x\in\mathsf{H}\;:  \langle x,a\rangle=t\Big\}\cup \Big\{\infty\Big\}.
\end{equation}.

$\bullet$  A reflection in  a  M\"{o}bius sphere $S$ is a transformation in  ${\widehat{\H}}$  which is either:	
\begin{eqnarray*}
 \rho(x)= a+\frac{r^{2}(x-a)}{|x-a|^{2}},\;\; \rho(a)=\infty\;\texttt{and} \;\; \rho(\infty)=a,
\end{eqnarray*}
if $S$ is of type $S(a,r)$, or:

	\begin{eqnarray*}
\rho(x)= x+\frac{2(t-\langle a,x\rangle)}{|a|^{2}}a,  \; \rho(\infty) =\infty,
\end{eqnarray*}
if $S$ is of type $ P(a,t)$.

$\bullet$ An orthogonal transformation of $\H$ is a linear map  $\o:\H\ap\H$ such that
 \begin{equation*}
  |\o(x)-\o(y)|=|x-y|  \quad  \textrm{for all} \ x,y\in\mathsf{H}.
\end{equation*}

$\bullet$ A similitude in  ${\widehat{\H}}$ is a transformation $\s$ such that
\begin{equation*}
\s(x)=\a\o(x)+a, \;\; \rho(\infty)=\infty
\end{equation*}
where $\o$ is an orthogonal transformation of $\H$, $\a\in \R$ and a fixed $a\in\H$.

\begin{Def} \cite{La}
A {\bf M\"{o}bius transformation} of $\widehat{\H}$ is a  bijection on  ${\widehat{\H}}$, which is a finite composition of reflections and similitudes.
\end{Def}

We have then the following characterizations:

\newpage
\begin{The}\cite{La}
\begin{enumerate}
\item A  bijection  $\phi$ of  ${\widehat{\H}}$  is a  M\"{o}bius transformation  of $\widehat{\H}$ if and only if
 the  image and the  inverse image by $\phi$ of M\"{o}bius spheres are M\"{o}bius spheres.
 \item A M\"{o}bius transformation $\phi$ is a similitude if and only if $\phi(\infty)=\infty$.
 \end{enumerate}
\end{The}

 Among the set of reflections, we have a particular one which is the reflection in $S(0,1)$ {\it i.e.}

 $\rho_0(x)=\dis\frac{x}{|x|^{2}}$ and $ \rho_0(0)=\infty, \; \rho_0(\infty)=0.$

 We have then:

 \begin{Pro}\cite{Be1} Let  $\mu$  be a  M\"{o}bius transformation. If $\mu(S(0,1))=S(a,r)$ then $\mu\circ \rho_0\circ\mu^{-1}$ is a reflection in $S(a,r)$. If  $\mu(S(0,1))=P(a,t)\cup\Big\{\infty\Big\}$
 then, $\mu\circ \rho_0\circ\mu^{-1}$ is the reflection in $P(a,t)$.
\end{Pro}

 Given a  M\"{o}bius sphere  $S$, if $S= S(a,r)$  the two sets

$ S^{-}(a,r)=\Big\{x\in\mathsf{H}:|x-a|^{2}<r^{2}\Big\}$

 $S^{+}(a,r)=\Big\{x\in\mathsf{H}:|x-a|^{2}>r^{2}\Big\}\cup\Big\{\infty\Big\} $

 are called the {\bf  two sides} of $S$. In the same way,  if $S= P(a,t)=\Big\{x\in\mathsf{H}|<x,a>=t\Big\}\cup \Big\{\infty\Big\}$,
the sets :

 $P^{-}(t,a)=\Big\{x\in \mathsf{H} :\langle a,x\rangle<t\Big\}$

 $P^{+}(t,a)=\Big\{x\in \mathsf{H} :\langle a,x\rangle>t \Big\}$

are the two sides of $P(a,t)$.
\begin{Pro}\cite{Be1} \label{side}
Let  $S_1$ and $S_2$ be the two sides of a M\"{o}bius sphere $S$. If $\mu$ is a M\"{o}bius transformation, then $\mu(S_1)$ and $\mu(S_2)$ are the two sides of the M\"{o}bius sphere $\mu(S)$. Moreover, if $\S$ is one side of $S$ then $\mu(\S)=\S$ implies $\mu(S)=S$.
\end{Pro}

Let $\{x_1,\cdots,x_n\}$ be a family of linearly independent vectors in $\H$ and $x \in\H$. An $n-${\it hyperplane} $P_n$  in $\H$  is a set of type:

$\Big\{x+\l_1x_1+\cdots+\l_nx_n, \;\l_i\in \R,\; i=1,\cdots, n \Big\}.$

A {\bf M\"{o}bius $n-$sphere} is an extended $n-$hyperplane  $P_n\cup\Big\{\infty\Big\}$ or a set of type $P_{n+1}\cap S(a,r)$ where $P_{n+1}$ is an $(n+1)-$hyperplane which contains $a$.

\begin{Pro}\cite{Be1}\label{nS}
For any M\"{o}bius transformation, the image of a M\"{o}bius $n-$sphere is a   M\"{o}bius $n-$sphere.
\end{Pro}

 {\it From now to the end  of this subsection}, we fix  the basis $\{e_i\}_{i\in I}$   in $\mathbb{H}$ and $\H$ is the hyperplane which is orthogonal to $e_1$. Each $x\in \mathbb{H}$ will be written $x=(x_1,\bar{x})$  with $ \bar{x} \in \H$. We denote by $\H^+=\{x\in {\mathbb{H}}\;: x_1>0\}$. Note that
$\H^+$ is one side of the M\"obius sphere   $\widehat{\H}$.

We denote by $\mathfrak{M}(\mathbb{H})$ the group of all M\"obius transformations of $\widehat{\mathbb{H}}$ such that $\mu(\H^+)=\H^+.$ Then, from Proposition \ref{side}, for $\mu\in \mathfrak{M}(\mathbb{H})$, we have $\mu(\widehat{\H})=(\widehat{\H})$.\\
The  converse is also true:

If $\mu$ is a reflection of $\widehat{\H}$ on ${P}(a,t)$,  let  $\tilde{\mu}$ be  the reflection  in $\widehat{ \mathbb{H}}$ on $\hat{P}((0,a),t)$  in $\widehat{ \mathbb{H}}.$

If $\mu$ is a reflection of $\widehat{\H}$ on $S(a,r)$, let  $\tilde{\mu}$ be the reflection  in $\widehat{ \mathbb{H}}$ on $\tilde{S}((0,a),r)$ in $\widehat{ \mathbb{H}}.$

 If $\mu=\a\o+a$ is a similitude of $\widehat{\H}$, let  $\tilde{\mu}=\a\tilde{\o}+(0,a)$ be the similitude in $\widehat{ \mathbb{H}}$ where $\tilde{\o}_{|\H}=\o$ and $\tilde{\o}(e_1)=e_1.$

In any case $\tilde{\mu}$ preserves $\H^+$ and $\H$. It follows that the group $\mathfrak{M}(\H)$ of M\"obius transformations of $\widehat{\H}$ is isomorphic to $\mathfrak{M}(\mathbb{H}).$

On the other hand, on $\mathbb{H}$, we consider the hyperbolic  distance $\d$ characterized by (cf \cite{Be1})
$$\cosh \d(x,y)=\sqrt{1+|x|^2}\sqrt{1+|y|^2}-<x,y> \textrm{ and } \d(x,y)\geq 0.$$

\begin{Def}\label{hyperbolic} A bijection $\phi$ of $\mathbb{H}$ is called a hyperbolic transformation if we have:
$$\forall\ x,y\in \mathbb{H};\ \  \d(\phi(x),\phi(y))=\d(x,y).$$
\end{Def}
Consider the diffeomorphism $h:\H^+\ap \mathbb{H}$ defined by
$$h(x_1,\bar{x})=(\dis\frac{ |x|^2-1}{2x_1},\dis\frac{\bar{x}}{x_1}).$$
The link between hyperbolic transformations and M\"obius transformations is given in the following result of \cite{Be2}.
\begin{The}\label{mobs}\cite{Be2}
\begin{enumerate}
\item   The group $\mathfrak{G}(\mathbb{H})$ of hyperbolic transformations of $\mathbb{H}$ is the set  $\{\phi=h\circ\mu\circ h^{-1} \;\;: \mu\in \mathfrak{M}(\mathbb{H})\}$.
\item Each map $\phi\in \mathfrak{G}(\mathbb{H})$ can be written as a similitude $\b$ or a product $\a\circ\rho_0\circ\b$ with $\a$ and $\b$ are of the form :\\
(i)  $\a(x)=k x +v \textrm{ with }  k>0,\; v\in \H $;\\
(ii) $\b(x)=k'\o(x)+v'\textrm{ with }  k'>0,\; v'\in \H,\; \o \textrm{ an  orthogonal transformation  of $\mathbb{H}$ such that } \o(v')=v'.$
\end{enumerate}
\end{The}
{\bf From now on, we identify  the groups $\mathfrak{M}(\mathbb{H})$ and $\mathfrak{G}(\mathbb{H})$}.\\
\begin{Rem}\label{restrict}${}$
\begin{enumerate}
\item According to \cite{Be2}, the pair $(\H^+,\mathfrak{M}(\mathbb{H}))$ is called the Poincar\'e model of hyperbolic geometry. In fact, let $g_{\H^+}=\frac{1}{x_1}g$ be the conformal metric to the canonical Riemannian metric where $g$ is induced by the inner product $<\;,\;>$ on $\mathbb{H}$. Then the map $h:(\H^+, g_{\H^+})\ap (\mathbb{H},\d)$ is an isometry.
\item If $\mathbb{H}$ is finite dimensional, each isometry is a bijection, but it is no longer true in general, if $\mathbb{H}$ is infinite dimensional (see \cite{Be1}).
\end{enumerate}
\end{Rem}
%%%%%%%%%%%%%%%%%%%%%%%%%%%%%%%%%%%%%%%%%%%%%%%%%%%%%%%%%%%%%%%%%%%

 %%%%%%%%%%%%%%%%%%%%%%%%%%%%%%%%%%%%%%%%%%%%%%%%%%%%%%%%%%%%%%%%%%
%%%%%%%%%%%%%%%%%%%%%%%%%%%%%%%%%%%%%%%%%%%%%%%%%%%%%%%%%%%%%%%%%%%%%%%%%%%%%%%%%%%%%%%%%%%%%%
\subsection{ M\"{o}bius transformations  and the Lorentz group }\label{Alorentz}${}$\\
%%%%%%%%%%%%%%%%%%%%%%%%%%%%%%%%%%%%%%%%%%%%%%%%%%%%%%%%%%%%%%%%%%
%%%%%%%%%%%%%%%%%%%%%%%%%%%%%%%%%%%%%%%%%%
In this subsection, we consider $\mathcal{H}= \R\oplus  \mathbb{H}$,  the basis $\{e_i\}_{i\in I}$ is fixed and again $\H$ is the orthogonal of $e_1$ in $\mathbb{H}$.  We put on $\cal H$ the  following Lorentz product:
$$<(s,x),(t,y)>_L=<{x},{y}>-st.$$
We then denote by  $|\;\;|_L$ the associated pseudo-norm and by ${\cal K}$ the {\it light cone} {\it  i.e.} \\ ${\cal K}=\{u=(s,x)\in {\cal H}\;: <u,u>_L=0\},$
and ${\cal K}^+=\{u=(s,x)\in {\cal K}\;: s>0\}.$

\begin{Def} \label{Lorentz} A bijection $\l$ of  ${\cal H}$ is called a Lorentz transformation if we have
$$\forall u,v\in\mathcal{H},\;|\l(u)-\l(v)|_L=|u-v|_L. \; $$
\end{Def}
On the other hand, we  consider
 the hyperboloid ${\cal H}_1=\{ u=(s,{x})\in {\cal H}\;: |u|_L^2=-1,\; \}$
and its "positive time like sheet"
${\cal H}_1^+=\{ u=(s,{x})\in {\cal H}_1\;: s>0\; \}$. Let $g: \mathbb{H}\rightarrow {\cal H}_1^+$  be a bijection defined by :$g(x)=(\sqrt{1+|{x}|^2}, x).$\\

The link between the Lorentz  transformations of $\cal H$ and  the hyperbolic transformations of $\mathbb{H}$ is given by the following  result (cf \cite{Be1}).

\begin{The}\label{hyplor} ${}$\\
 Given any hyperbolic transformation $\phi$, there exists a unique Lorentz transformation $\l=\t(\phi)$ such that
$$\l(0)=0\;,\; \l({\cal H}_1^+)={\cal H_1^+}\; and \;\;\forall x\in \mathbb{H},\; g(\phi(x))=\l(g(x)).$$
Moreover the restriction  to ${\cal H}_1^+$ of  the  Lorentz transformation $\t(\phi)$ associated to $\phi$ is given by:
$$\t(\phi)_{| {\cal H}_1^+}=g\circ \phi\circ g^{-1}.$$
\end{The}

According to this result, the Lorentz transformation  of type  $\l=\t(\phi)$, where $\phi$ is a hyperbolic transformation,  is then a continuous linear map which is called an {\it orthochronous Lorentz linear map}. \\

The set $SO(\mathbb{H},1)$ of linear Lorentz transformations $\l$ such that $\l({\cal K}^+)={\cal K}^+$ is a subgroup of the group $O(\mathbb{ H},1)$   of all linear Lorentz transformations and the set $SO_0(\mathbb{H},1)$ of orthochronous Lorentz linear maps is a subgroup of  $SO(\mathbb{H},1)$. Moreover, according to Theorem \ref{mobs}, Remark  \ref{restrict} and Theorem \ref{hyplor} we have a natural isomorphism $\mathcal{L}$ from the group of M\"{o}bius transformations $\mathfrak{M}(\mathbb{H})$ and the  group $SO_0(\mathbb{H},1)$. More precisely we have:
$${\cal L}( \l)= (g\circ h)^{-1}\circ \l_{| {\cal H}_1^+}\circ (g\circ h).$$

In fact,  as ${\cal H}_1^+$ is the  set $\{(s,x)\in {\cal H} \textrm{ such that } s=\sqrt{1+|{x}|^2}\}$ and so ${\cal H}_1^+$ is a smooth hypersurface.
In the restriction to ${\cal H}_1^+$ we have $<(s,x),(t,y)>_L=<x,y>-\sqrt{1+|x|^2}\sqrt{1+|y|^2}$. Therefore, in the restriction to ${\cal H}_1^+$  $\cosh \d(s,x),(t,y)=-<(s,x),(t,y)>_L$ defines a hyperbolic distance and the map $g(x)=(\sqrt{1+|{x}|^2}, x)$  is a diffeomorphism from $\mathbb{H}$ to ${\cal H}_1^+$ which is an isometry. According to Theorem \ref{mobs} and Theorem  \ref{hyplor} we get an natural identification of the group $\mathfrak{M}(\mathbb{H})$ and the group of the restriction to ${\cal H}_1^+$  of elements of  $SO_0(\mathbb{H},1)$.\\

\bigskip

{\it We end this subsection by recalling  a characterization of the group $SO(\mathbb{H},1)$ and its Lie algebra  (cf \cite{Ga} or  \cite{La}). We adopt the presentation of \cite{La}.}\\

According to the decomposition  ${\cal H}=\R\oplus\mathbb{H},$ let $p_1$ be (resp. $p_2$) the natural projection of $\cal H$ onto $\R$ (resp. $\mathbb{H}$). It follows that  each continuous linear map $A$ of $\cal H$ in a obvious  matrix form
\begin{eqnarray}\label{Adecompmat}
\begin{pmatrix} c&[v]^*\\
[u]&B\\
\end{pmatrix}
\end{eqnarray}
 where $c=p_1(A(1,0))$, $\;B=p_2\circ A_{| \mathbb{H}}$ and $u$ (resp $v$) is an element of $\mathbb{H}$ such that  $p_2\circ A(1,0)=u$ and $[u](s)=su$ (resp. $p_1\circ A(0,x)=<v,x> $ and $[v]^*(x)=<v,x>$).\\

Now, let  $J$ be  the continuous endomorphism of $\cal H$ defined by $J(s,x)=(-s,x)$. Given a continuous endomorphism $A$ of $\cal H$, the pseudo-adjoint $A^\#$ is the continuous endomorphism characterized by
$$<Au,v>_L=<u,A^{\#}v>_L \textrm{ for any } u,v\in{ \cal H}.$$
Thus, $A$ belongs to $O(\mathbb{H},1)$ (resp. $SO(\mathbb{H},1)$)  if and only if
$A^{\#}A=Id$  (resp. $A^{\#}A=Id$ and $A^{\#}\in SO(\mathbb{H},1)$). \\

According to the matrix form (\ref{Adecompmat}) , $A^\#$ has a matrix form of type

\begin{eqnarray}\label{pad}
\begin{pmatrix} c&-[u]^*\\
-[v ]&B^*\\
\end{pmatrix}
\end{eqnarray}
where $B^*$ is the adjoint endomorphism (of $\mathbb{H}$)  of $B$.
Then $A$ belongs to $O(\mathbb{H},1)$ if and only if
 \begin{eqnarray}\label{cractOH1}
\begin{pmatrix} c&-[u]^*\\
-[v]&B^*\\
\end{pmatrix}\begin{pmatrix} c&[v]^*\\
[u]&B\\
\end{pmatrix}=\begin{pmatrix} -1&0\\
0&Id\\
\end{pmatrix}
\end{eqnarray}
where $Id$ is the identity in $\mathbb{H}$. Moreover, $A\in  O(\mathbb{H},1)$ belongs to $SO(\mathbb{H},1)$ if and only if $c>0$ (see \cite{La}).

 The following result is   classical  in the finite dimensional case  and in the infinite dimensional  case it is more or less included in \cite{Be1} or \cite{La}
\begin{Pro}\label{decompA}${}$\\
 Let $A\in  O(\mathbb{H},1)$, there exists, $v\in\mathbb{H}$ with $v\not=0$,   such that  $A$ has the  following decomposition:
\begin{eqnarray}
A=PT
\end{eqnarray}
where $P=\begin{pmatrix} \varepsilon&0\\
0&Q\\
\end{pmatrix} $ and $Q^{-1}=Q^*$, $\; \varepsilon=\pm 1$ and  $T$ is such that:

 if $\H_v$ is the orthogonal of $\R.v$ in $\mathbb{H}$ then $T_{| \H_v}=Id_{\H_v}$ and  $T(\R\oplus \R.v)=\R\oplus \R.v$. Moreover, there exists $\a\geq 0$ such that the eigenvalues of  $T_{|\R\oplus \R.v}$ are $e^{\a}$ and $e^{-\a}$ with associated eigenvectors $(\frac{v}{|v|},1)$ and $(\frac{v}{|v|},-1)$ respectively.

\end{Pro}

Note that, in the previous decomposition, $T$ is called a {\it Lorentz boost} and it is characterized by $u\in \mathbb{H}$ and $\a>0$    so it  will denoted by  $B_{u,\a}$. Moreover according to Theorem \ref{hyplor}, $T$ is associated to a {\it hyperbolic translation} generated by $v$. (cf \cite{Be1}). Note that if  $\{u_i\}_{i\in I, i>1}$ is an orthonormal basis of $\H_v$, let  $Q$ be the linear isometry  in $\mathbb{H}$ such that $Q(e_1)=\dis\frac{v}{|v|}$ and $Q(e_i)=u_i, i\in I, i>1$. Then we have:

\begin{eqnarray}\label{decompoT}
B_{u,\a}=\begin{pmatrix} 1&0\\
0&Q\\
\end{pmatrix}\begin{pmatrix} \cosh \a&\sinh \a&0\\
\sinh \a& \cosh \a&0\\
0&0&Id_{\H_v}\\
\end{pmatrix}\begin{pmatrix} 1&0\\
0&Q^*\\
\end{pmatrix}.
\end{eqnarray}

Thus  we get the following corollary  (see also \cite{Be1}):\\
\begin{Cor}\label{lordecompo}${}$\\
For any  $A\in  O(\mathbb{H},1)$ there exists $Q$ and $Q'$ in $SO(\mathbb{H})$ and $\a>0$ such that
$$A= \begin{pmatrix} \varepsilon&0\\
0&Q'\\
\end{pmatrix}\begin{pmatrix} \cosh \a&\sinh \a&0\\
\sinh \a& \cosh \a&0\\
0&0&Id_{\H}\\
\end{pmatrix}\begin{pmatrix} 1&0\\
0&Q^*\\
\end{pmatrix}.$$

\end{Cor}

\begin{Rem} \label{hyptrans} According  to \cite{Be1} and our identifications, any boost  is a hyperbolic translation. Moreover, as in the finite dimension, in the metric space $({\H^+}, g_{\H^+})$ (cf Remark \ref{restrict} (1)), any boost $B_{e_1,\a}$ corresponds to the homothety $x\ap e^\a.x$ in $\H^+$ and so to the M\"obius transformation $x\ap e^\a.x$ in $\widehat{\H}$.
\end{Rem}
\bigskip
The proof of Proposition \ref{decompA} is an adaptation to our context of comparable result of  the finite dimensional case in \cite{Ga}

\begin{proof}  According to (\ref{Adecompmat}),  (\ref{pad}) and (\ref{cractOH1}), we get:
$$ B^*B=Id_\mathbb{H}+[v][v]^*\;\;\; [u]^*[u]=c^2-1\;\;\; [u]^*B=c[v]^*\;\;\;B^*u=cv$$
and also
$$ BB^*=Id_\mathbb{H}+[u][u]^*\;\;\; [v]^*[v]=c^2-1\;\;\;[v]^*B=c[u]^*\;\;\;Bv=cu.$$

On one hand, we get  as $ [v]^*[v]=|v|^2$ so $c^2=1+|v|^2$ and $c^2=1+|u|^2$ in particular $u\not=0$.
On the other hand the kernel of $[v][v]^*$ is the orthogonal $\H_v$ of $\R.v$ in $\mathbb{H}$. It follows that  the restriction of $[v][v]^*$  to $\H_v$ is zero and the restriction $[v][v]^*$ to $\R.v$ is such that $[v][v]^*(v)=|v|^2.v=(c^2-1)v$.
We deduce that $(Id_\mathbb{H}+[v][v]^*)_{| \H_v}=Id_{\H_v}$ and $v$ is an eigenvector of $Id_\mathbb{H}+[v][v]^*$ with eigenvalue $c^2$ of multiplicity $1$. From the polar decomposition theorem in Hilbert space, there exists a linear isometry $Q$ of $\mathbb{H}$ and a self-adjoint positive definite operator $S$ (on $\mathbb{H}$ ) such that $B=QS$. Moreover, we have $B^*B=S^2$ and so, $S_{| \H_v}=Id_{\H_v}$ and $S(v)=\pm c v$.  We may assume that this eigenvalue  $c$ is positive after changing
eventually $c$ into $-c$. Therefore we have $S(v)=cv.$\\

Assume at first that $c>0$. Since  $Bv=cu$, then $QS(v)=cQ(v)=cu$ and  so $Q(v)=u$.  We get
\begin{eqnarray}\label{decompocanonA}
\begin{pmatrix} c&[v]^*\\
[u]&B\
\end{pmatrix}=\begin{pmatrix} c&[v]^*\\
Qv&QS\\\end{pmatrix}=\begin{pmatrix} \varepsilon&0\\
0&Q\\
\end{pmatrix}\begin{pmatrix} c&[v]^*\\
[v]&S\\
\end{pmatrix}
\end{eqnarray}
with $c=\sqrt{|v|^2+1}$ and $\varepsilon=1$. We set $T=\begin{pmatrix} c&[v]^*\\
[v]&S\\
\end{pmatrix}$.\\

If $c<0$ by an analogue argument we get a decomposition as  (\ref{decompocanonA}) but with $\varepsilon=-1$.\\

Now, the restriction of $T$ to $\H_v$ is $Id_{\H_v}$  and, (in $\cal H$), $T(\R\oplus \R.v)=\R\oplus \R.v$. By similar arguments used in the proof of Proposition 2.4 of \cite{Ga} we complete the proof.

\end{proof}
\bigskip
 {\it In the sequence we denote by $\sqrt{Id_{\mathbb{H}}+[v][v]^*}$ the operator $S$ and so we have}
  \begin{eqnarray}\label{notT}
T=\begin{pmatrix} c&[v]^*\\
[v]&\sqrt{Id_\mathbb{H}+[v][v]^*}\\
\end{pmatrix} \textrm{ and } A=\begin{pmatrix} \varepsilon&0\\
0&Q\\
\end{pmatrix}  \begin{pmatrix} c&[v]^*\\
[v]&\sqrt{Id_\mathbb{H}+[v][v]^*}\\
\end{pmatrix}.
\end{eqnarray}

\bigskip

 Assume now that $I=\{1,\cdots,n\}$. According to  Proposition \ref{decompA} (see \cite{Ga})  any matrix $A\in O(n,1)$ can be written as a
  product of matrices of the form
 $$\begin{pmatrix}\varepsilon&0\\
 			0&Q\\
			\end{pmatrix}\begin{pmatrix}c&[v]^*\\
									[v]&\sqrt{Id_n+[v ].[v]^*}\\
									\end{pmatrix}$$
									
\noindent where $Q$ belongs to $O(n)$, $[v]$ is a vector column of $\mathbb{H}$   and $c=\sqrt{|v|^2+1}$ and $\varepsilon=\pm 1$.\\

Thus, the Lie group $O(n,1)$ has  $4$ connected components, according to the previous decomposition, we have  $\det Q=\pm 1$ and $\varepsilon=\pm1$. The group of Lorentz transformations is $SO(n,1)$ which is the group  corresponding to $\det Q=\varepsilon=\pm1$. According to the previous Proposition and   Theorem \ref{mobs},    the group $\mathfrak{ M}(\H)$  is isomorphic  to $SO(n,1)$, and so the group  $\mathfrak{ M}^+ (\H)$ which preserves the orientation is isomorphic to the connected components of the Identity in $SO(n,1)$, that is the subgroup $SO_0(n,1)$ corresponding to the case $\det Q=\varepsilon=1$.  \\

On the other hand  (see \cite{Ga} for instance), the Lie algebra $\mathfrak{so}(n,1)$ of $SO_0(n,1)$ is the set of matrices of the form
$$\begin{pmatrix}0&[u]^* \\
			[u]&B\\
			\end{pmatrix}$$
\noindent where $B$ is a square matrix of dimension $n$ such that $B^*=-B$. Therefore we have a natural decomposition
$$ \mathfrak{so}(n,1)=\mathfrak{h}_n\oplus \mathfrak{s}_n$$
where
 $$\mathfrak{h}_n=\Big\{\begin{pmatrix} 0& [u]^*\\
							[u]&0\\
							\end{pmatrix} \textrm{ where } [u] \textrm{ vector column } \in \R^n \Big\}$$
 $$\mathfrak{s}_n=\Big\{\begin{pmatrix} 0& 0\\
							0& B\\
							\end{pmatrix} \; B^*=-B\Big\}.$$

The vector space $\mathfrak{h}_n$ is generated by  $U_i=\begin{pmatrix} 0& [e_i]^*\\
							[e_i]&0\\
							\end{pmatrix}$ for $i=1,\cdots, n$
and  $\mathfrak{s}_n$ is a Lie  subalgebra of $\mathfrak{so}(n,1)$ generated by $\O_{ij}=\begin{pmatrix} 0& 0\\
							0& \o_{ij}\\
							\end{pmatrix}$ $1\leq i<j\leq n$,   where $\o_{ij}$ is the matrix with the term of index $ij$ (resp. $ji$) is $1$ (resp $-1$) and the other terms are $0$.\\

 \begin{Rem}\label{infiniteI} ${}$
 \begin{enumerate}
 \item When $I=\N$, the group $O(\mathbb{H},1)$ is a Lie subgroup of the group $GL({\cal H})$ of continuous automorphism of $\cal H$. However, this group has only two connected components and in particular $SO_0(\mathbb{H},1)=SO(\mathbb{H},1)$. On the other hand, in the decomposition \ref{notT}, $T$ belongs to $SO(\mathbb{H},1)$ so $A$ in (\ref{notT}) belongs to $SO(\mathbb{H},1)$ if and only if $\varepsilon=1$.\\
The Lie algebra $\mathfrak{so}(\mathbb{H},1)$ of $SO(\mathbb{H},1)$ has also a decomposition of type $\mathfrak{h}\oplus \mathfrak{s}$ where $ \mathfrak{h}$ is the set of endomorphism of type $\begin{pmatrix} 0& [u]^*\\
							[u]&0\\
							\end{pmatrix}$
where $u\in\mathbb{H}$ and $\mathfrak{s}$ is a Lie algebra that is, the set of endomorphisms of type $\begin{pmatrix} 0& 0\\
							0& B\\
							\end{pmatrix}$
where $B^*=-B$.\\ In fact $\mathfrak{s}$ is isomorphic to the Lie algebra of the group of linear isometry of $\mathbb{H}$ (cf \cite{La}).\\

\item  Consider the exponential map $\mathrm{Exp}:\mathfrak{so}(\mathbb{H},1)\ap SO(\mathbb{H},1)$. When $I=\{1,\cdots, n\}$, each boost $T$ can be written as $\mathrm{Exp}U$, for some $U\in\mathfrak{h}_n$ (cf \cite{Ga} for instance). On the other hand, each $P\in SO(n)$ can also be written as $\textrm{Exp}\O$ for some $\O$ of the Lie algebra of $SO(n)$. This implies that each element of $SO(n,1)$ can be written as $\textrm{Exp}\O\textrm{Exp}(U)$ for some $\O\in\mathfrak{s}_n$  and $U\in\mathfrak{h}_n$. Unfortunately $\O$ and $U$ do not commute and so $\textrm{Exp}(\O)\textrm{Exp}(U)\not= \textrm{Exp}(\O+U)$ and we do not get the surjectivity property of $\textrm{Exp}$. However,\\
    $\mathrm{Exp}:\mathfrak{so}(n,1)\ap SO_0(n,1)$ is surjective (see \cite{Ga} section 4.5).
\end{enumerate}
\end{Rem}
				
\smallskip

%%%%%%%%%%%%%%%%%%%%%%%%%%%%%%%%%%%%%%%%%%%%%%%%%%%%%%%%%%%%%%%%%%%%%%%%%%%%%%%%%%%%%%%%
\section{ The Hilbert-Schmidt M\"obius group of the unit sphere of $\mathbb{H}$}\label{SubRM}
%%%%%%%%%%%%%%%%%%%%%%%%%%%%%%%%%%%%%%%%%%%%%%%%%%%%%%%%%%%%%%%%%%%%%%%%%%%%%%%%%%%%%
%%%%%%%%%%%%%%%%%%%%%%%%%%%%%%%%%%%%%%%%%%%%%%%%%%%%%%%%%%%%%%%%%%%%%%%%%%%%%%
\subsection{Hilbert-Schmidt group of orthochronous Lorentz transformations}${}$\\
%%%%%%%%%%%%%%%%%%%%%%%%%%%%%%%%%%%%%%%%%%%%%%%%%%%%%%%%%%%%%%%%%%%%%%%%%%%%%%%%%%
%%%%%%%%%%%%%%%%%%%%%%%%%%%%%%%%%%%%%%%%%%%%%%%%%%%%%%%%%%%%%%%%%%%%%%%%%%%%%%%%%%
{\it Given a Hilbert space $\mathbb{H}$, we first  recall  results of  \cite{Go},  about some particular Lie sub-algebras of $L(\mathbb{H})$  of  the Lie algebra $L(\mathbb{H})$ of bounded operators on $\mathbb{H}$.}\\

\textrm{We consider a family $(G_n)_{n\in \N}$ of  connected finite dimensional Lie subgroups of $GL(\mathbb{H})$ such that
$$G_1\subset G_2\subset\cdots\subset G_n\subset\cdots\subset GL(\mathbb{H})$$
where $GL(\mathbb{H})$ denote the group of invertible elements of $L(\mathbb{H})$.\\
Let $\mathfrak{g}_n$  be the Lie algebra of $G_n$ and $\mathfrak{g}=\dis\bigcup_{n\in \N}\mathfrak {g}_n$. Then $\mathfrak{g}$ is a Lie algebra.}\\
\bigskip
\begin{Ass}\label{comp} There exists a subspace ${\mathfrak{g}}_\infty$ in $L(\mathbb{H})$ which contains $\mathfrak{g}$ and such that we can extend the inner product $<\;,\,>$  on $\mathfrak{g}$ to an inner product $<\;,\;>$ on ${\mathfrak{g}}_\infty$,  which is complete for the associated norm $|\;\;|$ and such that  $\mathfrak{g}$ is dense in ${\mathfrak{g}}_\infty$. Moreover, we assume that ${\mathfrak{g}}_\infty$ is closed under Lie bracket  of $L(\mathbb{H})$ and there
 exists a constant $C>0$  such that
 \begin{eqnarray}\label{majLie}
|[A,B]|\leq C| A|. | B|.
\end{eqnarray}
\end{Ass}
\bigskip
Let $C^1_{\mathfrak{g}}$ be the set of  piecewise $C^1$ paths $\g$ from $[0,1]$ to the Banach manifold  $GL(\mathbb{H})$  such that
\begin{center}
$\g'=\g^{-1}\circ \dot{\g}$ belongs to ${\mathfrak{g}}_\infty$ and $\g'$ is piecewise continuous for the norm $|\;|$ (on ${\mathfrak{g}}_\infty$).\\
\end{center}
On $GL(\mathbb{H})$ we define:
\begin{center}
$d(A,B)=\inf \Big\{\dis\int_0^1| \g'(s)| ds\;:\; \g\in  C^1_{\mathfrak{g}}\; \textrm{ such that }\; \g(0)=A,\; \g(1)=B\Big \}$
$d(A,B)=\infty \textrm{ if there is no }\g\in  C^1_{\mathfrak{g}}\; \textrm{ such that }  \g(0)=A,\; \g(1)=B.\;\;\;\;\;\;$
\end{center}

\begin{The}\cite{Go}\label{Liealgresult} Under the previous assumptions we have
\begin{enumerate}
\item Let $G_\infty=\{A\in GL(\mathbb{H})\;: d(A,Id_\mathbb{H})<\infty\}$. Then $G_\infty$ is a subgroup of $GL(\mathbb{H})$ and $d$ is a distance on this set which is left invariant.
\item  For the topology associated to $d$ the group  $G_{\infty}$ is closed,  and the group $G=\dis\bigcup_{n\in \N}G_n$ is dense in $G_{\infty}$.
\item  Let $d_n$ be the distance associated to the norm $|\;|$ on $\mathfrak{g}_n$. Then the distance  $d_{\infty}=\dis\inf _{n\in \N}d_n$ on $G$ coincides with the restriction of $d$.
\item The exponential map $\mathrm{Exp}:\mathfrak{g}_\infty\ap G_{\infty}$  is a local  diffeomorphism around $0$ in ${\mathfrak{g}}_\infty$.
\end{enumerate}
In particular, $G_\infty$ is a Lie group modeled on the Hilbert space $\mathfrak{g}_\infty$.
\end{The}

\textrm{The group $G_\infty$ is called a}  {\bf Cameron-Martin group} (cf \cite{Go}).\\

{\it From now to the end of this subsection,   we fix a Hilbert basis $\{e_i\}_{i\in \N\setminus\{0\}}$ of $\mathbb{H}$ and ${\cal H}=\R\oplus\mathbb{H}$ is  now equipped with the Hilbert inner product $<(s,x),(t,y) >=st+< x,y>$}.\\

 We can identify $\mathcal{H}$ with $l^2(\N)$. Let  $L_{HS}(\mathcal{H})$ be the subspace of Hilbert-Schmidt operators of $\mathcal{H}$, that is
$$L_{HS}(\cal H)=\{A\in L(\mathcal{H}) \textrm{ such that } \dis\sum_{i\in \N} |A e_i|^2<\infty\}.$$
Recall that on $L_{HS}(\mathcal{H})$ we have an inner product
$$< A,B>_{HS}=\dis\sum_{i\in \N} <Ae_i,B e_i>$$
and the  associated norm is
$$| A|_{HS}= (\dis\sum_{i\in \N} |A e_i|^2)^\frac{1}{2}.$$
Note that $L_{HS}(\cal H)$ is then a Hilbert space.\\
We can consider each operator $A\in L_{HS}(\mathcal{H})$  as  an infinite matrix $A=(a_{ij})_{i,j\in \N}$ such that $\dis\sum_{i,j\in \N}| a_{ij}|^2<\infty$. Therefore, if $ e_{ij}$  denote the infinite matrix  defined by:
\begin{center}
  $1$ at the $ij$th place and $0 $ at all other places,
\end{center}
 we get an  orthonormal basis $\{e_{ij}\} $ of  $L_{HS}(\mathcal{H})$ (relative to the inner product $<\;,\;>_{HS}$).  Note that $L_{HS}(\mathcal{H})$ is a Banach algebra (without unit)  for  the norm $|\;|_{HS}$ (cf \cite{Si}). In the Banach Lie group  $GL(\mathcal{H})$ of invertible bounded operators,  the general  Hilbert-Schmidt group is
$$GL_{HS}(\mathcal{H})=\{U\in L(\mathcal{H}) \textrm{ such that } Id_\mathcal{H}-U\in L_{HS}(\mathcal{H})\}.$$

On the other hand, denote by $\mathbb{H}_{n}$ the vector space generated by $\{e_1,\cdots,e_{n}\}$, and $\mathcal{H}_n$ the vector space $ \R\oplus \mathbb{H}_n$.

Now, we can  identify $L(\mathcal{H}_n)$ with the set
$$L_n(\mathcal{H})=\{A\in L_{HS}(\mathcal{H})\;:\mathcal{H}_n^\perp \subset \ker A \textrm{ and } \im A \subset \mathcal{H}_n\}.$$
 Since we have $\mathcal{H}_n\subset \mathcal{H}_{n+1}$, we have  $\mathcal{H}_{n+1}^\perp\subset\mathcal{H}_{n} ^\perp$ so, if $A\in L_n(\mathcal{H})$  then $A$ belongs to $L_{n+1}(\mathcal{H})$. In this way we obtain  an ascending  family:
\begin{eqnarray}\label{asL}
L_1(\mathcal{H})\subset L_2(\mathcal{H})\subset\cdots\subset L_n(\mathcal{H})\subset\cdots\subset L_{HS}(\mathcal{H})\subset L(\mathcal{H}).
\end{eqnarray}
 \noindent In the same way, we can identify $GL(\mathcal{H}_n)$ with the set
 $$GL_n(\mathcal{H})=\Big\{A\in GL_{HS}(\mathcal{H}) \textrm{ of type } \begin{pmatrix}Id_{\mathcal{H}_n^\perp}& 0\\
 																0& \bar{A}\\
																\end{pmatrix}\bar{A}\in GL(\mathcal{H}_n)\Big\}$$

and  by the similar arguments, we have also an ascending family
\begin{eqnarray}\label{asGL}
GL_1(\mathcal{H})\subset GL_2(\mathcal{H})\subset\cdots\subset GL_n(\mathcal{H})\subset\cdots\subset GL_{HS}(\mathcal{H})\subset GL(\mathcal{H}).
\end{eqnarray}
If $A$ belongs to $GL_{HS}(\mathcal{H})$ then the determinant of $A$ is well defined and $\det(A)\not=0$. Moreover,  according to the previous construction, any $A\in GL_{HS}(\mathcal{H})$ induces  a natural  endomorphism  $A_n\in GL_n(\mathcal{H})$.    We have then (cf \cite{Wi})
\begin{eqnarray}\label{det}
\det (A)=\dis\lim_{n\ap \infty}\det(A_n)
\end{eqnarray}

Now,  modulo the previous identification and according to the end of subsection \ref{Alorentz}, the family $(\mathfrak{so}(n,1))_{n\in\N}$  becomes a family of Lie subalgebras of $L_{HS}(\mathcal{H})$
  and the family of  Lie groups $(SO_0(n,1))_{n\in \N}$ becomes a family of ascending  Lie
subgroups of $GL(\mathcal{H})$ whose Lie algebras is the family $(\mathfrak{so}(n,1)))_{n\in\N}$ .

According to the end of subsection \ref{Alorentz} and the previous  notations, let  $U_i\in L_{HS}(\mathcal{H}) $ such that $U_i=e_{0i}+e_{i0}$ for $i\in \N\setminus\{0\}$ and $\O_{ij}=e_{ij}-e_{ji}$ for $0<i<j,\;, i,j\in \N$. We denote by $\mathfrak{h}_\infty\subset L_{HS}(\mathcal{H}) $ the Hilbert space generated by $\{U_i\}_{i\in \N\setminus \{0\}}$ and $\mathfrak{s}_\infty\subset  L_{HS}(\mathcal{H}) $ the Hilbert space generated $\{\O_{ij}\}_{0<i<j,\;, i,j\in \N}$. We set $\mathfrak{g}_\infty=\mathfrak{h}_\infty\oplus \mathfrak{s}_\infty$, according to the identification of  $L({\cal H}_n)$ with $L_n({\cal H})$, we can consider $\mathfrak{so}(n,1)$ as a subspace of $\mathfrak{g}_\infty$.\\

 From  Theorem \ref{Liealgresult}  we will deduce the following:

\begin{Pro}\label{CMMob}${}$
\begin{enumerate}
\item The vector space  $\mathfrak{g}_\infty$ is  the closure of $\mathfrak{g}=\dis\bigcup_{n\in \N} \mathfrak{so}(n,1)$ in $L_{HS}({\cal H})$. Moreover $\mathfrak{g}_\infty$  is Lie subalgebra of  $L_{HS}({\cal H})$  which satisfies the assumption \ref{comp}.
\item The Cameron-Martin group  ${G}_\infty$ associated to the ascending sequence $({SO}_0(n,1))_{n\in \N}$ in $L(\mathcal{H})$ is a Lie subgroup of $GL_{HS}(\mathcal{H})$ and $\dis\bigcup_{n\in \N}{SO}_0(n,1)$ is dense in ${G}_\infty$. Moreover, $\mathfrak{g}_\infty$ is the  Lie algebra of  ${G}_\infty$.
\item Each element  $A$ of $G_\infty$ can be written as $A=PT$  where $T$ is a boost and
 $P=\begin{pmatrix} 1&0\\
0&Q\\
\end{pmatrix} $ with  $Q^{-1}=Q^*$ and $\det( Q)=1$. In particular,  $SO(\mathbb{H},1)\cap GL_{HS}({\mathcal H})$ has two connected components and $G_\infty$ is the connected component of $Id_\mathcal{H}$.
\item The map $\mathrm{Exp}:\mathfrak{g}_\infty \ap G_{\infty}$ is a surjective  local diffeomorphism around  $0\in \mathfrak{g}_\infty$ .
\end{enumerate}
\end{Pro}

\begin{Rem}\label{relwithSO}${}$\\
Note that the Lie group $SO(\mathbb{H},1)$ is connected (cf Remark \ref{infiniteI} part (1)) while $SO(\mathbb{H},1)\cap GL_{HS}({\cal H})$ has two connected components.
\end{Rem}

\begin{Def} The sub-group $G_\infty$ of $SO(\mathbb{H},1)$ built in Proposition \ref{CMMob} is  called the Hilbert-Schmidt orthochronous Lorentz  group and  will be denoted $SO_{HS}(\mathbb{H},1)$. The corresponding Lie algebra $\mathfrak{g}_\infty$ will be  denoted $\mathfrak{so}_{HS}(\mathbb{H},1)$.
\end{Def}

In the remaining part of the article , we simply denote by  $\mathfrak{h}$ (resp. $\mathfrak{s}$) each subspace $\mathfrak{h}_\infty\subset \mathfrak{g}_\infty$ (resp.   $\mathfrak{s}_\infty\subset \mathfrak{g}_\infty$ ) and so we get
\begin{eqnarray}\label{decomposo}
\mathfrak{so}_{HS}(\mathbb{H},1)=\mathfrak{h}\oplus\mathfrak{s}.
\end{eqnarray}

If we now consider the natural isomorphism ${\cal L}:SO(\mathbb{H},1)\ap \mathfrak{M}(\mathbb{H})$ (cf subsection \ref{Alorentz}), we get a subgroup $\mathfrak{M}_{HS}(\mathbb{H})={\cal L}(SO_{HS}(\mathbb{H},1))$ of $ \mathfrak{M}(\mathbb{H})$. In this way,  $\mathfrak{M}_{HS}(\mathbb{H}) $ can be provided with a Lie group structure and its Lie algebra $\mathfrak{m}_{HS}(\mathbb{H})$ is isomorphic to $\mathfrak{so}_{HS}(\mathbb{H},1)$. \\

\begin{Def} The group  $\mathfrak{M}_{HS}(\mathbb{H}) $ is called the Hilbert-Schmidt  group of M\"obius transformations of $\mathbb{H}.$
\end{Def}

In finite dimension, in \cite{GaXu},  the authors  gives a complete description of the  map $\textrm{Exp}:\mathfrak{so}(n)\ap SO(n)$.  Using  similar results  in  an infinite dimensional Hilbert space context, we obtain:

\begin{The} \label{prodexp}${}$\\
 Consider $\mathrm{Exp}:  \mathfrak{so}_{HS}(\mathbb{H},1)\ap SO_{HS}(\mathbb{H},1)$ and fix some $A=PT\in SO_{HS}(\mathbb{H},1)$.  According to (\ref{decomposo}), there exists $U\in\mathfrak{h}$, a family $\{B_j\}_{j\in J}\subset \mathfrak{s}$ with $J\subset \N$ of finite rank   and a non increasing  sequence $(\theta_j)_{j\in J}$ of real numbers with $0<\theta_j\leq\pi$
with the following properties
\begin{enumerate}
% \item[(i)]$A_j=\begin{pmatrix} 1&0\\
%0&B_j\\
%\end{pmatrix} $
\item[(i)] $[B_k,B_j]=0$ for $k\not=j$,
\item[(ii)] $A=\dis\prod_{j\in J}\textrm{Exp}(\theta_j B_j)\textrm{Exp U}$.
 \end{enumerate}
\end{The}

\smallskip

{\it As the proof of the Theorem \ref{prodexp}, is technical and has no  direct relation with the context of  M\"obius transformation, we will give its proof in Appendix \ref{A1}.}
 \smallskip

\begin{proof}[Proof of Proposition \ref{CMMob}]${}$\\ According to Theorem \ref{Liealgresult}, we have only to prove that    $\mathfrak{g}_\infty$    satisfies the assumption \ref{comp}.
At first, by construction,  as $\mathfrak{so}(n,1)$ is a subset of  $L_n({\cal H})$, for each $n\in\N\setminus \{0\}$, $\mathfrak{so}(n,1)$ is generated by $\{U_{i}\}_{1\leq i\leq n},\{\O_{ij}\}_{1<i<j\leq n}$ so, $\mathfrak{g}=\dis\bigcup_{n\in \N} \mathfrak{so}(n,1)$ is dense in  $\mathfrak{g}_\infty$. Also by construction,
the natural inner product on $L_n({\cal H})$ which is isometric to the canonical inner product of $L({\cal H}_n)$ so that $\{e_{ij}\}_{i,j\in \N\setminus\{0\}}$ is the canonical orthonormal basis. It follows that  $\mathfrak{g}_\infty$ is a closed subspace of $L_{HS}({\cal H})$, which is provided with an inner product extends the inner product on each $ \mathfrak{so}(n,1)$. On the other hand, by an elementary  calculation, according to the Lie bracket $[A,B]=AB-BA$ on $L({\cal H})$ we have the following relations:
\begin{eqnarray}\label{lierelations}
[U_i,U_j]=\O_{jk},\;\; [U_i,\O_{jl}]=\d_{ij}U_l-\d_{il}U_j,\;\; [\O_{ij},\O_{kl}]=\d_{il}\O_{jk}+\d_{jk}\O_{il}-\d_{ik}\O_{jl}-\d_{jl}\O_{ik}.
\end{eqnarray}
It follows that   $\mathfrak{g}_\infty$ is closed under the Lie bracket of $L_{HS}({\cal H})$. It remains to show that relation (\ref{majLie}) is satisfied for any $A$ and $ B$ in  $\mathfrak{g}_\infty$.
%Any $A\in \mathfrak{g}_\infty$ can be written (using Einstein  convention):
%$$A=u^iU_i+a^{ij} \O_{ij}$$
According to (\ref{lierelations}), the definition of $U_i$ and $\O_{ij}$, and the fact that $\{e_{ij}\}_{i,j\in \N}$ is an orthonormal basis in $L_{HS}({\cal H})$  we have the following majorations:
\begin{eqnarray}\label{majU}
|[U_i,U_j]|_{HS}\leq  2,\;\;\;\;\; |[U_i,\O_{jk}]|_{HS}\leq 4,\;\;\;\; \;|[\O_{ij},\O_{kl}]|_{HS}\leq 8.
\end{eqnarray}
Now, any $A\in \mathfrak{g}_\infty$ can be written (using Einstein  convention):
$$A=u^iU_i+a^{ij} \O_{ij},$$
so $|A|^{2}_{HS}=2(\dis\sum_{i\in \N}(u^i)^2+\dis\sum_{0< i<j, i,j\in \N} (a^{ij})^2)$. According to the bi-linearity of $[\;,\;]$, relations  (\ref{lierelations}) and (\ref{majU}) we easily  get a relation of type (\ref{majLie}) for the Lie bracket on $\mathfrak{g}_\infty$.\\
The other properties  in (1) and (2)
are direct consequences of Theorem \ref{Liealgresult}.\\

 Any  $M\in GL_{HS}({\cal H})$ induces a natural element $M_n\in GL_n({\cal H})$, and of course,  $GL_{HS}({\cal H})$ is the Cameron-Martin group associated to the ascending family (\ref{asGL}). In particular, according to  the notations of  Theorem \ref{Liealgresult} , we have:
 \begin{eqnarray}\label{dG}
  \dis\lim_{n\ap \infty}d_{\infty}(M,M_n)=0.
\end{eqnarray}

Now, Let $A\in G_\infty$. As $G_\infty\subset SO(\mathbb{H},1)$, according to Proposition \ref{decompA}, we can write $A=PT$ where $T$ is a boost and  $P=\begin{pmatrix} \varepsilon&0\\
0&Q\\
\end{pmatrix} $ with  $Q^{-1}=Q^*$.   With the previous convention, for each $n$, we have  $A_n=P_nT_n$  where $T_n$ is a boost in ${\cal H}_n$ and  $P_n=\begin{pmatrix} \varepsilon&0\\
0&Q_n\\
\end{pmatrix} $ with  $(Q_n)^{-1}=(Q_n)^*$. By construction of $G_\infty$,  $A_n$ belongs to $SO_0(\mathbb{H}_n,1)$ so $\varepsilon =1$ and $\det(Q_n)=1$. From (\ref{dG}),  in  $P$ we must have $\varepsilon =1$ and  $\det(Q)=1$. The same arguments applied to $A\in SO(\mathbb{H},1)\cap GL_{HS}({\mathcal H})$  implies that  $A=PT$ with  $P=\begin{pmatrix} \varepsilon&0\\
0&Q\\
\end{pmatrix}$ and $\det(Q)=\varepsilon=\pm1$. This ends Part (3).\\

 As $\textrm{Exp}:\mathfrak{g}_\infty\ap G_\infty$ is a smooth map, Part (4) is then a consequence of Point  (2) of Remark \ref{infiniteI}

  and the construction of $G_\infty$.
\end{proof}

\smallskip

%%%%%%%%%%%%%%%%%%%%%%%%%%%%%%%%%%%%%%%%%%%%%%%%%%%%%%%%%%%%%%%%%%
\subsection{ Hilbert-Schmidt M\"obius group of the unit sphere of $\mathbb{H}$}\label{actS}${}$\\
%%%%%%%%%%%%%%%%%%%%%%%%%%%%%%%%%%%%%%%%%%%%%%%%%%%%%%%%%%%%%%%%%
%%%%%%%%%%%%%%%%%%%%%%%%%%%%%%%%%%%%%%%%%%%%%%%%%%%%%%%%%%%%%%%%%%

Given a  Hilbert basis $\{e_i\}_{i\in I}$ we again denote by  $\H$ the orthogonal of $e_1$.  Consider any $v\in \mathbb{H}$ with $v\not=0$ and $\H_v$ the orthogonal of $\R.v$ in $\mathbb{H}$.   If $e_1$ and $v$ are linearly independent,  after  changing $v$ into $-v$ if necessary, we may assume that $<v,e_1>=v_1\geq 0$ so $v$ belongs to $\H^+=\{x\;,: x_1\geq 0\}$.  If we set $e=\dis\frac{v}{|v|}$,  we have an orthogonal isometry $R_v$ such that  $R_v(e)=e_1$ and then $R_v(\H_v)=\H$. We get an isomorphism from the group $\mathfrak{M}(\H)$ to the group $\mathfrak{M}(\H_v)$ of M\"obius transformations of $\widehat{\H_v}=\H_v\cup\{\infty\}$. Then we identify these groups.\\

In $\mathbb{H}$ we consider the unit sphere
  $ \Sp_\mathbb{H}=\Big\{z\in \mathbb{H}, |z|=1\Big\}$ and  the point $N=(1,\bar{0})$. The stereographic projection (cf  \cite{Be1}) is the map:
$$ {\Pi}: \Sp_\mathbb{H}\backslash\{N\} \longrightarrow \H
\qquad\qquad  (x_1,\bar{x}) \longmapsto \frac{\displaystyle \bar{x}}{\displaystyle 1-x_1}.
$$

We can extend ${\Pi}$  to $ \Sp_\mathbb{H}$ into $\widehat{\H}$ by setting ${\Pi}(1,\bar{0})=\infty$. Then  $\Pi$  becomes an homeomorphism from $ \Sp_\mathbb{H}$ to $\widehat{\H}$, whose inverse  is the map
$$\bar{x} \longmapsto \Big(\dis\frac{|\bar{x}|^2-1}{|\bar{x}|^2+1},\dis\frac{2\bar{x}}{|\bar{x}|^2+1}\Big) \textrm{ and } \infty \longmapsto N.$$

\begin{Def}\label{mobS}${}$\\ A diffeomorphism $\phi$ of  $ \Sp_\mathbb{H}$ is called a M\"obius transformation of  $ \Sp_\mathbb{H}$ if $\Pi\circ\phi\circ \Pi^{-1}$ belongs to $\mathfrak{M}(\H)$.
\end{Def}

The {\it  group of M\"obius transformations of  $ \Sp_\mathbb{H}$} is denoted  $\mathfrak{M}(  \Sp_\mathbb{H})$. Thus, modulo a choice of a Hilbert basis we get an  isomorphism ${\cal P}: \mathfrak{M}(  \Sp_\mathbb{H})\ap \mathfrak{M}_{HS}(\H)$. Let  $\mathfrak{M}_{HS}(\Sp_\mathbb{H}) $ be  the subgroup associated  $\mathfrak{M}_{HS}(\H) =\{\mu_{| \H},\; \mu\in\mathfrak{M}_{HS}(\mathbb{H})\} $  via the isomorphism $\cal P$. This group will be called the {\it Hilbert-Schmidt M\"obius group} of  $ \Sp_\mathbb{H}$. The Lie algebra $\mathfrak{m}_{HS}(\Sp_\mathbb{H})$ of this group is then isomorphic to $\mathfrak{g}_\infty$.\\

\noindent Consider  $ v\in \mathbb{H}$ with $v\not=0$.
 There exists $R_v\in O(\mathbb{H})$ such that $R_v(\H_v)=\H$  (see the beginning of this subsection) and  so  we have $\Pi_v=R^{*}_v\circ \Pi$, $R^{*}_v$ is the adjoint of $R_{v}$ . It follows that $\phi$ belongs to $\mathfrak{M}(  \Sp_\mathbb{H})$ if and only if $\Pi_v\circ\phi\circ\Pi_v^{-1}$ belongs to $\mathfrak{M}(\H_v)$ and then our definition of $\mathfrak{M}(  \Sp_\mathbb{H})$ is independent of the choice of the basis $\{e_i\}_{i\in I}$ of $\mathbb{H}$.\\

Now, the unit sphere $\Sp_{\mathbb{H}}$ is a Hilbert submanifold of $\mathbb{H}$, and the tangent space $T_z \Sp_{\mathbb{H}}$ at  $z\in\Sp_{\mathbb{H}}$ can be identified with the hyperplane $\H_v$.   We denote by $g_{{\Sp_{\mathbb{H}}}}$ the Riemannian metric  on $T\Sp_{\mathbb{H}}$  induced by $< \;,\;>$. Let  $\varphi_v$ be the function on $\Sp_\mathbb{H}$ defined by $\varphi_v(x)=<\frac{v}{|v|},x>$. The gradient of $\varphi_v$ (relative to the Riemannian metric  $g_{\Sp_{\mathbb{H}}}$) is   the vector field on ${\Sp_{\mathbb{H}}}$   defined by:
\begin{eqnarray}\label{grad}
\textrm{grad} (\varphi_v)(x)= \frac{v}{|v|}-<\frac{v}{|v|},x> x.
\end{eqnarray}

Let  $\d_v$ be the dilation of $\H_v$ of coefficient $e^{t.| v|}$. According to  Remark \ref{hyptrans}, for any $v\in \mathbb{H}\setminus\{0\}$ and $t\in \R$, the  family of transformations
$$\G^v_t(x)=((\Pi_v)^{-1}\circ \d_v\circ \Pi_v)(x)$$
 is a one-parameter family of M\"obius transformations of $\Sp_{\mathbb{H}}$.\\

 Following on the steps of \cite{Ha}, we have
 \begin{Pro}\label{grafG}${}$
 \begin{enumerate}
 \item[(i)] For $t$ fixed, each M\"{o}bius transformation $\G^v_t$ belongs to $\mathfrak{M}_{HS}(\Sp_\mathbb{H}) $.
\item[(ii)] Let $\Phi^v_t$ be  the flow of  $\textrm{grad}(\varphi_v)$. Then we have $\Phi^v_t=\G^v_t$.
\item[(iii)]  For any  pair $v,w$ of independent vectors of $\mathbb{H}\setminus\{0\}$,  the flow generated by the Lie  bracket $[\textrm{grad}(\varphi_v),\textrm{grad}(\varphi_w)]$ is a rotation in the plane $P(v,w)$ generated by $v$ and $w$ with rotation angle of value $-t$.
\end{enumerate}
\end{Pro}
\begin{proof} If $I$ is finite, the proof is given in \cite{Ha} and \cite{Ro} so we assume that $I=\N$. Fix some $v\in \mathbb{H}$. We choose a Hilbert basis $\{e_i\}_{i\in \N}$ such that $e_1=\dis\frac{v}{|v|}$. Then we have $\H\equiv \H_v$ and $\Pi_v\equiv \Pi$. For each $n$ we denote by $\H_n$ the orthogonal of subspace $\{e_1\}$ in $\mathbb{H}_n$. By induction,  we can put on each  $\H_n$ an orientation such that the orientation given by $\H_n$ and $e_{n+1}$ is the orientation of $\H_{n+1}$. Since $\d_v$ preserves  the orientation in the restriction to any $\H_n$, it follows that    $(\Pi\circ\G^v_t\circ \Pi^{-1})$ preserves the orientation  of $\H_n$ and finally $[\Pi\circ\G^v_t\circ\Pi^{-1}]$ preserves the orientation for any $n$. Therefore , $A_n={\cal L^{-1}}\circ[\Pi\circ\G^v_t\circ\Pi^{-1}]_{|\H_n}$ belongs to $SO(\mathbb{H}_n,1)$. Moreover,  if $A={\cal L^{-1}}\circ[\Pi\circ\G^v_t\circ\Pi^{-1}]$, then  we have $[A]_{|{\cal H}_n}=A_n$. This implies that $A_n$  is a Cauchy sequence in $G_\infty$  for the distance $d_\infty$. We deduce that $A$ is the limit of $A_n$ and so $A$ belongs to $G_\infty$. This ends the proof of Part (i).\\

The proof of Part (ii) (resp. Part (iii)) is formally the same as the proof of Lemma 3.1 (resp. Lemma 3.3) of \cite{Ha} so we will give an abstract of these proofs. \\
 As $\G^{\l v}_t=\G^v_{\l t}$ without loss of generality we can assume that $|v|=1$. At first $x=\pm v$ are fixed points for $\Phi^v_t$ and $\G^v_t$. Pick some $z\in \Sp_{\mathbb{H}}$ with $z\not=\pm v$ and let  $P$ be  the plane in $\mathbb{H}$ generated by $v$ and $z$.  By similar arguments to those in the proof of Lemma 3.1   of \cite{Ha}, we have
 $\textrm{grad}\phi_v(x)$ belongs to $P$ for all $ x\in P$ and so $\Phi^v_t$ preserves $P$. On the other hand, by construction, $\G^v_t$ also preserves $P$.  Now, from Lemma 2.2 of \cite{Ha} we then get that $\G^v_t$ and $\Phi^v_t$
coincide on the circle $P\cap \Sp_\mathbb{H}$, so we get Part (ii).\\

Let $P$ be the plane generated by $v$ and $w$, where $v,w$ are independent vectors of $\mathbb{H}\setminus\{0\}$.  Since   $\Phi^v_t=\G^v_t$ and $\Phi^w_t=\G^w_t$ these flows  preserve $P$, so the Lie bracket $[\textrm{grad}(\varphi_v),\textrm{grad}(\varphi_w)]$ is tangent to $P$ on $P$. Therefore the flow of  $[\textrm{grad}(\varphi_v),\textrm{grad}(\varphi_w)]$ preserves $P$ and according to Lemma 2.2 of \cite{Ha} in restriction to $P$, this flow is a rotation with rotation angle of value $-t$. It remains to show that if $x\in \Sp_{\mathbb{H}}$ is orthogonal to $P$, this flow keeps $x$ fixed. It reduces to a $3$-dimensional problem which can be solved as in the proof of Lemma  3.1 in \cite{Ha}.

\end{proof}

 Now,  If  $\{e_i^*\}_{i\in I}$ is the dual basis of $\{e_i\}_{i\in I}$   the map $\varphi_{e_i}$ is exactly the  dual form $e_i^*$ and  we denote by $\xi_i$ the gradient of $e^*_i$. As vector field, we have the decomposition (see \cite{Ro} and \cite{PeSa}):
\begin{equation}\label{xi}
\xi_i(z)=\dis\frac{\p}{\p x_i}-z_i\dis\sum_{l\in I}z_l\dis\frac{\p}{\p x_l}.
\end{equation}

 Therefore the bracket $[\xi_i,\xi_j]$ has the decomposition:
 \begin{equation}\label{xixj}
 [\xi_i,\xi_j](z)=z_i\dis\frac{\p}{\p x_j}-z_j\dis\frac{\p}{\p x_i}.
\end{equation}
\newline
\bigskip

Consider the  natural action  $:\mathfrak{A}:\mathfrak{M}_{HS}(\Sp_{\mathbb{H}})\times\Sp_{\mathbb{H}}\ap \Sp_{\mathbb{H}}$ on $\Sp_{\mathbb{H}}$ and  we  denote by $\mathfrak{a}:\mathfrak{m}_{HS}(\Sp_{\mathbb{H}})\ap \textrm{Vect}(\Sp_{\mathbb{H}})$  the associated infinitesimal action where  $\textrm{Vect}(\Sp_{\mathbb{H}})$ is
 the space of vector fields on $\Sp_{\mathbb{H}}$. If we identify $
 \mathfrak{m}_{HS}(\Sp_{\mathbb{H}})$ with $\mathfrak{g}_\infty$, it is classical that we have (cf \cite{He} or \cite{Ro})
$$\mathfrak{a}([U_i,U_j]=-[\mathfrak{a}(U_i),\mathfrak{a}(U_j)].$$

As in finite dimension (cf \cite{Ro}) we have:

\begin{Pro}\label{action S}${}$
\begin{enumerate}
\item The action $ \mathfrak{A}$ is effective.\footnote{An action $\mathfrak{ A}$ is called effective  if ${\mathfrak{A}}(g , z) = z \;\;\forall z$  implies $g=Id$}
\item The morphism $ \mathfrak{a}$ is injective  and  $\mathfrak{a}(U_i)=\xi_i$.
\end{enumerate}
\end{Pro}

\begin{proof}${}$\\
(1) Let  $\phi\in \mathfrak{M}_{HS}(\Sp_{\mathbb{H}})$ such that $\phi(z)=z$ for all $z\in \Sp_{\mathbb{H}}$. According to Proposition \ref{nS},  for any $n$ the restriction of $\phi$ to $\mathbb{H}_n\cap \Sp_{\mathbb{H}}$  is a  M\"obius transformation of the finite dimensional sphere $\mathbb{H}_n\cap \Sp_{\mathbb{H}}$. As in the finite dimensional case this action is effective,  the restriction of $\phi$ to $\mathbb{H}_n\cap \Sp_{\mathbb{H}}$ is the identity. Therefore  the map $\Pi\circ\phi\circ\Pi^{-1}$ from $\H$ to $\H$ is the identity on each subspace $\H_{n-1}=\H\cap\mathbb{H}_n$ for any $n\in \N$. It follows that $\Pi\circ\phi\circ \Pi^{-1}=Id_\H$ and then $\phi=Id_{\Sp_{\mathbb{H}}}$. Therefore the action $\mathfrak{A}$ is effective.\\
(2) For the injectivity of $\mathfrak{a}$,  see the proof of Proposition 2.9 of \cite{Ro} part (5). On the other hand according to our identifications, from Proposition \ref{grafG}, we get   $\mathfrak{a}(U_i)=\xi_i$.\\
\end{proof}

%%%%%%%%%%%%%%%%%%%%%%%%%%%%%%%%%%%%%%%%%%%%%%%%%%%%%%%%%%%%%%%%%%%%%%%%%%%%%%%%%%
\subsection{On the sub-Riemannian structure on   $\mathfrak{M}_{HS}(\Sp_\mathbb{H}) $}\label{subG}${}$
%%%%%%%%%%%%%%%%%%%%%%%%%%%%%%%%%%%%%%%%%%%%%%%%%%%%%%%%%%%%%%%%%%%%
%%%%%%%%%%%%%%%%%%%%%%%%%%%%%%%%%%%%%%%%%%%%%%%%%%%%%%%%%%%%%%%%%%%%%%

Let $M$ be a Hilbert manifold and $\cal D$ a subbundle of $TM$. A {\it sub-Riemannian structure on $M$} is a triple $(M,{\cal D},g)$ where $g$ is a Riemannian metric on ${\cal D}$. Of course, given a Riemannian metric $\bar{g}$ on $M$, we get a Riemannian metric $g$ on $\cal D$ by restriction. On the other hand, there always exists a complementary ${\cal V}$ of $\cal D$, {\it i.e.}  $TM={\cal D}\oplus{\cal V}$ and so we can extend $g$ into a Riemannian metric $\bar{g}$ on $M$ in an evident way.\\

Consider any Riemannian metric  $\bar{g}$ on $M$.  A curve $\g:[0,T]\ap M$  is of class $L^1$ if we have:\\  $\dis\int_0^T\sqrt{\bar{g}(\dot{\g}(t),\dot{\g}(t))}dt<\infty$. This property does not depend on the choice of $\bar{g}$.  For such a curve $\g$, its length $l(\g)$ is precisely the quantity $\dis\int_0^T\sqrt{\bar g(\dot{\g}(t),\dot{\g}(t)}dt$  and, of course, $l(\g)$  does not depend on its parametrization.
A $L^1$- curve  is called {\it  horizontal} if $\dot{\g}(t)$ belongs to ${\cal D}(\g(t))$ .   Given any Riemannian metric $g$ on $\cal D$,   the length of an horizontal curve $\g$ is well defined. Note that we also have
\begin{eqnarray}\label{length}
l(\g)=\dis\int_0^T |\g(t)^{-1}\dot{\g}(t)|dt.
\end{eqnarray}

Given two points $x_0$ and $x_1$ in $M$, let ${\cal C}_H(x_0,x_1)$ be the set, eventually empty, of horizontal $L^1$-curves $\g:[0,T]\ap M$ such that $\g(0)=x_0$ and $\g(T)=x_1$ for some $T\geq 0$.  reparametrization,
The {\it horizontal distance} $d_H(x_0,x_1)$ between $x_0$ and $x_1$ is defined by
\begin{eqnarray}\label{Hdis}
d_H(x_0,x_1)=\dis\inf\big{\{}l(\g), \g\in {\cal C}_H(x_0,x_1)\big{\}} \textrm{ and }  d_H(x_0,x_1)=\infty \textrm{ if } {\cal C}_H(x_0,x_1)=\emptyset.
\end{eqnarray}
In the finite dimension, the infimum in (\ref{Hdis}) is always reached. Moreover, the Theorem of Chow gives sufficient conditions under which any two points of $M$ can be joined by a horizontal curve.  In this case,  $d_H$ becomes a distance.\\
In infinite dimension, as in the Riemannian case,  if  $ {\cal C}_H(x_0,x_1)\not=\emptyset$,  the infimum in (\ref{Hdis}) could  be not reached. Moreover, in this context, to our knowledge, no general result as Chow's theorem exists.  Therefore  we cannot hope that $d_H$ is a distance in a wide context.\\

We now come back to the Lie group $\mathfrak{M}_{HS}(\mathbb{S}_\mathbb{H})$. The Lie algebra $\mathfrak{m}_{HS}(\mathbb{S}_\mathbb{H})$ and $\mathfrak{g}=\mathfrak{so}_{HS}(\mathbb{H},1)=\mathfrak{h}\oplus\mathfrak{s}$  being identified and we  provide this Lie algebra with the norm $|\;|$ associated to the  inner product induced by $\dis\frac{1}{2}<\;,\;>_{HS}$. Then,  the isomorphism
$u\ap \begin{pmatrix} 0& [u]^*\\
					[u]&0\\
					\end{pmatrix}$ from $\mathbb{H}$ to $\mathfrak{h}$ is in fact an isometry.  For simplicity, the inner product on $\mathfrak{h}$ will be denoted $<\;,\;>$.   Therefore the  Hilbert subspace $(\mathfrak{h}, <\;,\;>)$ generates a left invariant distribution $\D$  and also a left  invariant Riemannian metric $g$ on $\D$  on  $\mathfrak{M}_{HS}(\Sp_\mathbb{H}) $ and then ( $\mathfrak{M}_{HS}(\Sp_\mathbb{H}) ,\D,g)$ is a sub-Riemannian structure on $\mathfrak{M}_{HS}(\Sp_\mathbb{H})$.  Given any $\phi\in\mathfrak{M}_{HS}(\Sp_\mathbb{H})$,  the {\it accessibility set} of  $\phi$ is

$${\cal A}(\phi)=\{\psi\in \mathfrak{M}_{HS}(\Sp_\mathbb{H}) \textrm{ such that there exists an horizontal curve }\g:[0,T]\ap  \mathfrak{M}_{HS}(\Sp_\mathbb{H}) \textrm{ with } \g(0)=\phi\;\; \g(T)=\psi\}.$$

On the other hand, in the Lie sub-algebra $\mathfrak{s}$ of $\mathfrak{g}$ of  $\mathfrak{M}_{HS}(\mathbb{S}_\mathbb{H})$, we consider the Banach space

$\mathfrak{s}_1=\{ P\in \mathfrak{s} \textrm{ such that } P=\dis\sum_{k,l\in I, k<l}\l_{kl}\O_{kl}, \dis\sum_{k,l\in I, k<l}|\l_{kl}|<\infty\}$

equipped with the norm $|P|_1=\dis\sum_{k,l\in I, k<l}|\l_{kl}|$. Note that $|P|_1=\dis\sum_{i\in I}|<e_i,Pe_i>|$ is the $L^1$ trace of $P$ and so $|P|_1$ does not depend on the choice of the fixed Hilbert  basis of $\mathbb{H}$. We denote by $\mathfrak{g}_1=\mathfrak{h}\oplus\mathfrak{s}_1$ equipped with the norm
$$||(B,P)||_1=|B|+|P|_1.$$
Of course, the natural inclusion of $\mathfrak{g}_1$ in $\mathfrak{g}$ is continuous,  and the family $\{U_i\}_{i\in I}\cup\{\O_{kl}\}_{k,l\in I,k<l}$ is a Schauder basis of $\mathfrak{g}_1$. We denote by $\mathfrak{M}^1_{HS}(\Sp_\mathbb{H})=\mathrm{Exp}(\mathfrak{g}_1)$. Then, it is clear that  $\mathfrak{M}^1_{HS}(\Sp_\mathbb{H})$ has a structure of a Banach Lie group modeled on $\mathfrak{g}_1$. Moreover, $\mathfrak{M}^1_{HS}(\Sp_\mathbb{H})$ is dense in $\mathfrak{M}_{HS}(\Sp_\mathbb{H})$. According to the terminology of weak submanifold of a Banach manifold (cf \cite{PeSa}), we will say that $\mathfrak{M}^1_{HS}(\Sp_\mathbb{H})$ is a weak Lie subgroup of $\mathfrak{M}_{HS}(\Sp_\mathbb{H})$. Then, we have:

\begin{The}\label{subR}${}$
\begin{enumerate}
\item[(i)] Any two elements $A_0$ and $A_1$ of  $\mathfrak{M}^1_{HS}(\Sp_\mathbb{H})$ can be joined by a horizontal curve.

\item[(ii)] $d_H$ is a distance on $\mathfrak{M}^1_{HS}(\Sp_\mathbb{H})$.
\end{enumerate}
\end{The}

\begin{proof}${}$\\
According to the construction of $\mathfrak{M}_{HS}(\mathbb{S}_{\mathbb{H}})$, we can assume that $\mathfrak{M}^1_{HS}(\Sp_\mathbb{H})=SO^{1}_{HS}(\mathbb{H},1)\subset SO_{HS}(\mathbb{H},1)$. On the other hand,
it is sufficient to prove part (i) for $A_0=Id$ and $A_1$ any point of $\mathfrak{M}^1_{HS}(\Sp_\mathbb{H})$. Fix some
$ A\in\mathfrak{M}^1_{HS}(\Sp_\mathbb{H})$.
According to Theorem \ref{prodexp}, we have $A=\dis\prod_{j\in J}\textrm{Exp}(\theta_j B_j)\textrm{Exp U}$  where each $B_j$ is a finite rank element of $\mathfrak{s}$ and $U\in \mathfrak{h}$. Therefore  the curve $t\ap \textrm{Exp}(tU)$ is a horizontal curve  defined on $[0,1]$ which joins $Id$ to  $\textrm{Exp}(U)$. It is sufficient to prove the result for  $A=\dis\prod_{j\in J}\textrm{Exp}(\theta_j B_j)$. We denote by $A_j=\textrm{Exp}(\theta_j B_j)$ and  by ${\cal B}$ a  matrix consisting of blocks $\bar{\cal B}_j=\theta_j B_j$ in restriction to $E_j$. Fix such a point $A_j$.  By construction of  $B_j$ (cf Appendix \ref{A1}), if we set  $\mathbb{E}_j=B_j(\mathbb{H})$, then $\mathbb{E}_j$   is a finite dimensional Hilbert space such that $\ker (B_j)=(\mathbb{E}_j)^\perp$. It follows that   %$B_j$
$\bar{B}_j={B_j}_{| \mathbb{E}_j}$
 belongs to $\mathfrak{so}(\mathbb{E}_j,1)$. Moreover, we also have a decomposition $\mathfrak{so}(\mathbb{E}_j,1)=\mathfrak{h}_j\oplus \mathfrak{s}_j$ and $\bar{B}_j$ belongs to $\mathfrak{s}_j$.  In the basis of $\mathbb{E}_j$ built in Appendix \ref{A1},  the Lie algebra  $\mathfrak{so}(\mathbb{E}_j,1)$ is generated by $\{\bar{U}_{r}={U_{rs}}_{|\mathbb{E}_j}, \; l_1\leq 2l_j\}$ and $\{\bar{\O}_{rs}={\O_{rs}}_{| \mathbb{E}_j}, l_1\leq r<s\leq l_j\}$.  Consider the  (left-invariant) sub-Riemannian structure  on ${SO}(\mathbb{E}_j,1)$
 generated by $\mathfrak{h}_j$, provided with   inner product such that $\{\bar{U}_r,\;\; r=1,\cdots n_j\}$ is orthonormal basis. According to the classical Chow theorem, there is a horizontal curve $\bar{\g}_j:[0,T_j]\ap {SO}(\mathbb{E}_j,1)$ such that $\bar{\g}_j(0)=Id_{\mathbb{E}_j}$ and $\bar{\g}_j(T_j)=\bar{A}_j={A_j}_{| \mathbb{E}_j}$. Consider  $\g_j:[0,T_j]\ap {SO}_{HS}(\mathbb{H} ,1)$ defined by $\g_j(t)_{|\mathbb{E}_j}=\bar{\g}_j(t) \; \textrm{ and } \g_j(t)_{|(\mathbb{E}_j)^\perp}=Id_{|(\mathbb{E}_j)^\perp}$.

Then $\g_j$ is a horizontal curve which joins $Id_{\mathbb{H}}$ to $A_j$.

If $J$ is finite, we can assume that $J=\{1,\cdots,N\}$ otherwise, we can assume that $J=\N.$\
We parameterize $\g_j$ into a curve $c_j$ on $[\t_{j-1},\t_j]$ by setting $c_j(s)= \g_j(s-\t_{j-1})$.

For each integer $n\in J$, consider the finite  composition $ C_n:[0,\t_n]\ap {SO}_{HS}(\mathbb{H} ,1)$  inductively defined by
$C_n(s)=C_{n-1}(s) \textrm{ for } s\in [0,\t_{n-1}]$,

$C_n(s)=c_n(s) C_{n-1}(\t_{n-1})  \textrm{ for } s\in[\t_{n-1},\t_n]$.

 Then $C_n$ is a $L^1$ horizontal  curve  which joins $Id_{\mathbb{ H}}$ to $\dis\prod_{j=1}^n A_j$. Therefore  if $J$ is finite the proof is complete. Assume now that  $J=\N$. We set $\t=\dis\lim_{n\ap \infty}\t_n$ if this limit is finite otherwise we set $\t=\infty$. We must show that $\lim_{n\ap \infty}C_n(s)$ is well defined for all $s\in [0,\t]$. At first, as $A=\lim_{n\ap \infty}\dis\prod_{j=1}^nA_j$ we have then
 $\lim_{n\ap \infty}C_n(\t_n)=A$. But, by construction,  for each $m>n$, we have ${C_m}_{| [0,\t_n]}=C_n$. So, for any $s\in [0,\t[$ there exists $n$ such that $s\in[0,\t_n]$ and so $C(s)=C_n(s)$ is well defined. Of course, such  a construction is differentiable almost every where but  without a good choice for each curve $\bar{\g}_j$, in general, $\t=\infty$ and even if $\t$ is finite, $C$ is  not of class $L^1$ and in particular, we can have

 $\dis\lim_{n\ap\infty}\int_0^{\t_n}\sqrt{g(\dot{C}_n(s),\dot{C}_n(s))}ds=\infty$.\\

 To end this proof, we will use the  results about the sub-Riemannian structure of $SU(1,1)$ to get the following Lemma.

 \begin{Lem}\label{bong} For each $j$, with the previous notations, we can choose an horizontal curve $\bar{\g}_:[0,T_j]\ap {SO}(\mathbb{E}_j,1)$  arc-length parametrized such that
 $$\bar{\g}_j(0)=Id_{\mathbb{E}_j},\;\;\;\bar{\g}_j(T_j)=\bar{A}_j,\;\textrm{ and } l(\bar{\g}_j)=n_j |\theta_j|=T_j.$$
 \end{Lem}

 For each $j\in J$,  $C_{n}(\t_{n})=\dis\prod_{j=1}^{n}A_j$.   Therefore , by construction of the family $A_j$,  the endomorphism $C_{n}(\t_{n})$  is an isometry of $\mathbb{H}$ which preserves the space
 $\mathbb{K}\oplus \mathbb{E}_1\oplus\cdots,\oplus \mathbb{E}_{n}$. Since $\g_{n+1}$ is arc-length parametrized  we have:
 \begin{eqnarray}\label{lengthCn}
\dis\int_{\t_n}^{\t_{n+1}}\sqrt{g(\dot{C}(s),\dot{C}(s))}ds=\dis\int_{0}^{T_{n+1}}\sqrt{g(\dot{\g}_{n+1}(s),\dot{\g}_{n+1}(s)) }ds=T_j=n_j|\theta_j|
\end{eqnarray}

 But  from the decomposition of ${\cal B}$,
 $$|{\cal B}|_1=2\dis\sum_{j\in J}n_j|\theta_j|.$$
and, according to (\ref{lengthCn}) and the construction of $C$ we have then
$$l(C)=\dis\int_{0}^\t\sqrt{g(\dot{C}(s),\dot{C}(s))}ds=\dis\sum_{j\in J}\dis\int_{\t_{j_1}}^{\t_j}\sqrt{g(\dot{C}(s),\dot{C}(s))}ds=\dis\frac{1}{2}|{\cal B}|_1.$$
This ends the proof of Part (i).\\

 From the definition of the distance $d$ on $\mathfrak{M}_{HS}(\Sp_\mathbb{H})$,  for any $\psi$ and $\psi'$ in  $\mathfrak{M}^1_{HS}(\Sp_\mathbb{H})$ we have $d_H(\psi,\psi')\geq d(\psi,\psi')$. As from part (i) the restriction of $d_H$ to $\mathfrak{M}^1_{HS}(\Sp_\mathbb{H})$ it follows easily that $d_H$ is a distance and then Part (ii) is proved.\\

\end{proof}

%%%%%%%%%%%%%%%%%%%%%%%%%%%%%%%%%%%%%%%%%%%%%%%%%%%%%%%%%%%%%%%%%%%%%%%
\section{Control problem of a Hilbert snake and accessibility sets  }\label{HilbSn}
%%%%%%%%%%%%%%%%%%%%%%%%%%%%%%%%%%%%%%%%%%%%%%%%%%%%%%%%%%%%%%%%%%%%%%%%%
%%%%%%%%%%%%%%%%%%%%%%%%%%%%%%%%%%%%%%%%%%%%%%%%%%%%%%%%%%%%%%%%%%%%%%%%%%%%%%%
\subsection{The configuration space}${}$\\
%%%%%%%%%%%%%%%%%%%%%%%%%%%%%%%%%%%%%%%%%%%%%%%%%%%%%%%%%%%%%%%%%%%%%%%%%%%%%%
%%%%%%%%%%%%%%%%%%%%%%%%%%%%%%%%%%%%%%%%%%%%%%%%%%%%%%%%%%%%%%%%%%%%%%%%%%%%%%%%
Again in this section,   the Hilbert  basis $\{ e_{i}\}_{i\in \mathbb{N}} $ in  $\mathbb{H}$ is fixed.\\
A curve $\gamma :\left[ a,b\right] \rightarrow \mathbb{H}$ (not necessary continuous) is  called ${C}^{k\text{ }}$-piecewise if there exists a finite set\\ $\mathcal{P}=\left\{ a=s_{0}<s_{1}<...<s_{N}=b\right\} $ such that, for all $i=0,...,N-1$,
the restriction of  $\gamma$ to the interval $]s_{i},s_{i+1}[$ can be extended to a curve  of class ${C}^{k}$ on the closed interval
$\left[ s_{i},s_{i+1}\right] $.
Given any metric space $(X,d)$  and partition $\mathcal{P}$=$
\left\{ a=s_{0}<s_{1}<...<s_{N}=b\right\} $ of $[a,b]$,
 let  $\mathcal{C}_{\mathcal{P}}^k\left( \left[
a,b\right] ,X\right)$ be  the set of curves $u:[0,L]\ap X$
 which are $C^k$-piecewise relatively to $\mathcal{P}$ for $k\in \N$ and equipped with the distance
 $$\d(u_1,u_2)=\dis\sup_{t\in [0,L]} d(u_1(t),u_2(t)).$$ Note that, if ${ \cal P}=\{0,L\}$ and if $X$ is a submanifold of $\mathbb{H}$, then $\mathcal{C}_{\mathcal{P}}^k\left( \left[
a,b\right] ,X\right)$ is the space of continuous ${\cal C}^k([0,L],X)$  curves from $[0,L]$ to $X$   of class $C^k$ , and as in finite dimension,  we have a natural structure of Banach manifold on  ${\cal C}^k([0,L],X)$.\\

\noindent{\it Throughout this paper, we fix a real number $L>0$  and  $\mathcal{P}$ is a given fixed partition of $[0,L]$.} \\

{\bf A Hilbert snake} is a continuous  piecewise $C^1$-curve $S : [0,L] \ap  \field{H}$, such that $||\dot{S}(t)||=1$ and S(0) = 0.
\noindent In fact, a snake is characterized by $u(t)=\dot{S}(t)$ and of course we have $S(t)=\dis\int_0^tu(s)ds$  where $u:[0,L] \ap  \field{S}^\infty$ is a piecewise $C^0$-curve associated to the partition $\cal P$. Moreover, this snake is affine  if and only if $u$ is constant on each subinterval of $\cal P$.
The set
${\cal C}^L_\mathcal{P}={\mathcal{C}}_{\mathcal{P}}^0\left(
\left[ 0,L\right] ,
\mathbb{S}_\mathbb{H}\right) $
is called the {\bf configuration space of the snake}  in $\mathbb{H}$ of length $L$ relative to the partition $\cal P$.

The map $ u\mapsto(u\mid_{[s_{0},s_{1}]},...,u\mid_{[s_{i},s_{i+1}]},u\mid_{[s_{N-1},s_{N}]})$ is an homeomorphism between  ${\cal C}^L_{\cal P} $ and\\ $\prod_{i=0}^{N-1}{\cal C}^{0}([s_{i},s_{i+1}],\field{S}^{\infty})$.  Moreover, this map  permits to put on  ${\cal C}^L_\mathcal{P}$  a structure of Banach manifold  diffeomorphic to the Banach product structure $\prod_{i=0}^{N-1}{\cal C}^{0}([s_{i},s_{i+1}],\field{S}^{\infty})$.

\noindent The tangent space $T_u{\cal C}^L_{\cal P}$
    can be identified with the set
$$\{ v\in {\cal C}^0_{\cal P}([0,L],\field{H}) \textrm { such that } <u(s),v(s)>=0 \textrm{ for all } s\in [0,L]\}.$$
 This space is naturally provided with two non equivalent norms

 the natural  $||\;||_\infty$

  the $||\;||_{L^2}$ associated to the inner product
 $<v,w>_{L^2}=\dis\int_{0}^L<v(s),w(s)>ds$.

%%%%%%%%%%%%%%%%%%%%%%%%%%%%%%%%%%%%%%%%%%%%%%%%%%%%%%%%%%%%
\subsection{The horizontal distribution associated to a Hilbert snake}\label{hori}${}$\\
%%%%%%%%%%%%%%%%%%%%%%%%%%%%%%%%%%%%%%%%%%%%%%%%%%%%%%%%%%%%%%
%%%%%%%%%%%%%%%%%%%%%%%%%%%%%%%%%%%%%%%%%%%%%%%%%%%%%%%%%%%%%%
For any $u\in {\cal C}^L_{\cal P}$  the {\bf Hilbert snake} associated to $u$ is the map $S_u:[0,L]\ap \field{H}$ defined  by
\begin{center}$S_u(t)=\dis\int_0^tu(s)ds$.  The {\bf endpoint map}:
${\cal E}:  {\cal C}^L_{\cal P} \ap \field{H}$ defined by
$u\ap S_u(L)$
\end{center}
 is smooth and we have $
T_u{\cal E}(v)=\dis\int_0^Lv(s)ds$.

Let ${\cal D}_u$  be the orthogonal of $\ker T_u{\cal E}$ (for the inner product $<\;,\;>_{L^2}$ on $T_u{\cal C}^L_{\cal P}$). Then we have the decomposition
 $$T_u{\cal C}^L_{\cal P}={\cal D}_u\oplus \ker T_u{\cal E}$$
and the restriction of $T_u{\cal E}$ to ${\cal D}_u$ is a continuous  injective morphism into $\field{H}$.
 The family $u \mapsto {\cal D} _u$ is a (closed) distribution on ${\cal C}^L_{\cal P}$ called  the {\bf  horizontal distribution},
and each vector field $X$ (resp. curve)  on ${\cal C}^L_{\cal P}$ which is tangent to ${\cal D}$ is called a {\bf horizontal vector field} (resp. {\bf horizontal curve}).

The inner product on $\field{H}$ gives rise to a Riemannian  metric $g$ on $T\field{H}\equiv \field{H}\times\field{H}$ given by $g_x(u,v)=<u,v>$.
Let $\phi:\field{H}\ap \R$ be a smooth function. The usual gradient  of $\phi$ on $\field{H}$ is the vector field
$$\mathrm{grad}(\phi)=(g^\flat)^{-1}(d\phi),$$
where $g^\flat$ is the canonical isomorphism of bundle  from $T\field{H}$ to its dual bundle $T^*\field{H}$, corresponding to the Riesz representation i.e.
$g^\flat(v)(w)=<v,w>$. Thus  $\mathrm{grad}(\phi)$ is characterized by:
\begin{eqnarray}\label{grad}
 g( \mathrm{grad}(\phi),v)=<\mathrm{grad}(\phi),v>=d\phi(v),
\end{eqnarray}
for any $v\in \field{H}$.

In the same way, to the inner product on   $T{\cal C}^L_{\cal P}$ (previously defined), is associated a   {\it weak} Riemannian metric ${G}$ and we cannot define in the same way the gradient of any smooth function on  ${\cal C}^L_{\cal P}$. However, let  \\${G}^\flat:T{\cal C}^L_{\cal P}\ap T^*{\cal C}^L_{\cal P}$ be the morphism bundle  defined by:
$${G}_u^\flat(v)(w)={G}_u(v,w)$$
for any $v$ and $w$ in $T_u{\cal C}^L_{\cal P}$.  Given any smooth function $\phi: \mathbb{H}\ap \R$,
then $\ker d(\phi\circ {\cal E})$  contains $\ker T{\cal E}$  and so belongs to ${G}_u^\flat(T_u{\cal C}^L_{\cal P})$. Moreover,
\begin{eqnarray}
\nabla\phi=({G}^\flat)^{-1}(d(\phi\circ{\cal E}))
\end{eqnarray}
 is tangent to ${\cal D}_u$, and we have
\begin{eqnarray}\label{expnabla}
\nabla\phi(u)(s)=\mathrm{grad}(\phi)({\cal E}(u))-<\mathrm{grad}(\phi)({\cal E}(u)),u(s)>u(s).
\end{eqnarray}

\noindent The vector field  $\nabla\phi$ is called  {\bf horizontal gradient} of $\phi$.

\noindent To each vector $x\in \field{H}$, we can associate the linear form $x^*$ such that $x^*(z)=<z,x>$. This implies that the horizontal gradient  $\nabla x^*$ is well defined. In particular,

\begin{Obs}\label{Ei}${}$\\
 To each vector  $e_i$, $i\in \N$, of the Hilbert basis, we can associate the horizontal vector field $E_i=\nabla e_i^*$.
  in fact,  the family $\{E_i\}_{i\in\N}$ of vector fields generates the distribution ${\cal D}$.
\end{Obs}

%%%%%%%%%%%%%%%%%%%%%%%%%%%%%%%%%%%%%%%%%%%%%%%%%%%%%%%%%%%%
\subsection{Set of critical values and set of singular points of the endpoint map}\label{sing}${}$\\
%%%%%%%%%%%%%%%%%%%%%%%%%%%%%%%%%%%%%%%%%%%%%%%%%%%%%%%%%%%%%%%%%%%%%%%%%%%%%%%%%%%%%
%%%%%%%%%%%%%%%%%%%%%%%%%%%%%%%%%%%%%%%%%%%%%%%%%%%%%%%%%%%%%%%%%%%%%%%%%%%%%%%%%%%%%%

As the continuous linear map $T_u{\cal E}: T_u{\cal C}^L_{\cal P}\ap T_{{\cal E}(u)}\field{H}\equiv \field{H}$ is closed
it follows that $\rho_u=T_u{\cal E}_{|{\cal D}_u}$  is  an isomorphism from ${\cal D}_u$ to  the closed subset $\r_u({\cal D}_u)$ of $\field{H}$.

Consider the decompositions $z=\dis\sum_{i\in \N}z_ie_i$ and $u(s)=\dis\sum_{i\in \N}u_i(s)e_i$.
Then $u$ is singular if and only if (cf \cite{PeSa}):
\begin{eqnarray}\label{sing2}
Lz_i=\dis\sum_{j\in \N}\int_0^Lu_i(s)u_j(s)z_jds\; \forall i\in \N.
\end{eqnarray}
Let $\G_u$ be the endomorphism  defined by matrix of general term $ (\int_0^Lu_i(s)u_j(s)ds)$. Note that $\G_u$ is self-adjoint. The endomorphism $A_u=L.Id-\G_u$ is also self-adjoint and, in fact,
 its matrix  in the basis $\{e_i\}_{i\in \N}$ is $ (L\d_{ij}-\int_0^Lu_i(s)u_j(s)ds)$. It follows that  (\ref{sing2}) is equivalent to
\begin{eqnarray}\label{sing3}
A_u(z)=0.
\end{eqnarray}
 Finally,  $u$ is a singular point if and only if $L$ is an eigenvalue of $\G_u$ and also  if and only if  the vector space generated by $u([0,L])$ is $1$-dimensional. \\

The image of $\cal E$ is the closed ball $B(0,L)$ in $\field{H}$ and  {\bf set of critical values} of $\cal E$ is the union of spheres $S(0,L_j)$ for $j=1,\cdots n$  with $0\leq L_j\leq L$.

Finally we obtain the following result (\cite{PeSa}):

\begin{Pro}\label{sub}${}$
\begin{enumerate}
\item The set  ${\cal R}({\cal E})$ (resp. ${\cal V({\cal E})}$) of regular values (resp. points) of $\cal E$ is an open dense subset of ${\cal C}^L_{\cal P}$ (resp. $\field{H}$).
\item
 For any $u\in{\cal R}({\cal E})$  the linear map $\r_u:{\cal D}_u\ap \{{\cal E}(u)\}\times \field{H}$   is an isomorphism, and  on ${\cal D}_u$,  the inner product induced by $<\;,\;>_{L^2}$ and the inner product defined  $\r_u$ from $\field{H}$ are equivalent.  Moreover the distribution  ${\cal D}_{|{\cal R}({\cal E})}$ is  a trivial  Hilbert  bundle over ${\cal R}({\cal E})$
 which  is isometrically  isomorphic  to $T\field{P}^\infty$.
\end{enumerate}
\end{Pro}

%%%%%%%%%%%%%%%%%%%%%%%%%%%%%%%%%%%%%%%%%%%%%%%%%%%%%%%%%%%%%%%%%%%%%%%%%%%%%%%%
%%%%%%%%%%%%%%%%%%%%%%%%%%%%%%%%%%%%%%%%%%%%%%%%%%%%%%%%%%%%%%%%%%%%%%%%%%%%%%%%%
\subsection{Accessibility  results for a Hilbert snake}\label{result}${}$\\
%%%%%%%%%%%%%%%%%%%%%%%%%%%%%%%%%%%%%%%%%%%%%%%%%%%%%%%%%%%%%%%%%%%%%%%%%%%%%%%%%%
%%%%%%%%%%%%%%%%%%%%%%%%%%%%%%%%%%%%%%%%%%%%%%%%%%%%%%%%%%%%%%%%%%%%%%%%%%%%%%%%%%%
Recall that  given any continuous  piecewise $C^k$-curve $c:[0,T]\ap \field{H}$, a  lift of $c$ is a continuous piecewise $C^k$-curve $\g: [0,T] \ap {\cal C}^L_{\cal P}$ such that
${\cal E}(\g(t))=c(t)$. Thus, for a Hilbert snake we can consider the following optimal  control problem :

Given any  continuous piecewise $C^k$-curve $c: [0,T] \ap\field{H}$, we look for  a lift  $\g:[0,1]\ap {\cal C} ^L_{\cal P}$, say  $ t\ap u_t$, such that, for all $t\in [0,1]$,

 -- the associated  family $S_t=\dis\int_0^L u_t(s)ds$ of snakes satisfies  $S_t(L)=c(t)$ for all $t\in [0,1]$,

 -- the {\bf infinitesimal kinematic energy}:
 $\dis\frac{1}{2} ||\dot{\g}(t)||_{L^2}=\dis\frac{1}{2} G(\dot{\g}(t),\dot{\g}(t))$
  is minimal.\\

Then such a type of optimal problem has a solution if and only if the curve $c$ has a horizontal lift. We shall say that such a horizontal lift is an {\bf optimal control}.\\

 On the other hand, we can also ask when two positions $x_0$ and $x_1$  of the "head" of the snake can be joined  by a continuous  piecewise smooth curve $c$ which has an optimal control ${\g}$ as lift. As in finite dimension, the {\bf accessibility set} ${\cal A}(u)$, for some $u\in{\cal C}^L_{\cal P}$,  is the set of endpoints $\g(T)$ for any piecewise smooth   horizontal curve $\g:[0,T]\ap{\cal C}^L_{\cal P} $ such that $\g(0)=u$. In this case if $x_0=S_u(L)$ then any $z=S_{u'}(L)$ can be joined from $x_0$ by an absolutely continuous  curve   $c$ which has an optimal control when $u'$ belongs to ${\cal A}(u)$.\\

When $\mathbb{H}$ is finite dimensional, the set ${\cal A}(u)$  is exactly the orbit of the action. In finite dimension, given any horizontal distribution ${\cal D}$ on a finite dimensional manifold $M$, the famous {\it Sussmann's Theorem} (see \cite{Su}) asserts that  each accessibility set is a smooth immersed manifold which is an integral manifold of a distribution $\hat{\cal D}$ which contains $\cal D$  (i.e. ${\cal D}_x\subset \hat{\cal D}_x$ for any $x\in M$) and  characterized by:

 $\hat{\cal D}$ is the smallest distribution which contains $\cal D$  and which is invariant by the flow of any (local) vector field tangent to $\cal D$.\\

 From this argument, E.  Rodriguez proved that  the set ${\cal A}(u)$ is an immersed finite dimensional submanifold of ${\cal C}^L_{\cal P}$  in \cite{Ro}.

In the context of Banach manifolds  the reader can find some generalization of this Sussmann's  result in \cite{LaPe}. Unfortunately,  in our context,   this last results  give only some    density results on accessibility sets, with analogue construction as in finite dimension case (see \cite{PeSa}). \\

 Precisely, according  to observation \ref{Ei}, to each Hilbert basis $\{e_i,\;i\in\N\}$   the family ${\cal X}=\{E_i,\;i\in \N\}$  of (global) vector fields on ${\cal C}^L_{\cal P}$ generates the horizontal distribution ${\cal D}$. On the other hand  the family
  $${\cal Y}={\cal X}\bigcup\{[E_i,E_j],\; i,j\in I, \;i<j\}$$
   generates a weak Hilbert  distribution $\bar{\cal D}$ on ${\cal C}^L_{\cal P}$.
   Then we have:

\begin{The}\label{accesPeSa}\cite{PeSa}${}$\\
The distribution $\bar{\cal D}$  has the following properties:

(i) $\bar{\cal D}$ does not depend on the choice of the basis  $\{e_i,\;i\in\N\}$;

(ii) $\hat{\cal D}_x$ is dense in $\bar{\cal D}_x$ for all $x\in M$;

(iii) $\bar{\cal D}$ is integrable;

(iv) the accessibility set ${\cal A}(u)$ of a point  $u$ of any maximal integral manifold $N$ of $\bar{\cal D}$   is a dense subset of  $N$.\\
\end{The}

In the following section we will give a new proof of this Theorem which use the natural action of $\mathfrak{M}_{HS}(\mathbb{S}_{\mathbb{H}})$ on ${\cal C}_{\cal P}^{L}$,  the sub-Riemannian  structure of $\mathfrak{M}_{HS}(\mathbb{S}_{\mathbb{H}})$ and  Theorem \ref{subR}. We also get a geometrical interpretation of the maximal integral manifold of $\cal D$

   \bigskip

 %%%%%%%%%%%%%%%%%%%%%%%%%%%%%%%%%%%%%%%%%%%%%%%%%%%%%%%%%%%%%%%%%%%%%%%%%%%%%
 \subsection{ Action of $\mathfrak{M}_{HS}(\mathbb{S}_{\mathbb{H}})$ on ${\cal C}_{\cal P}^{L}$  and proof of Theorem \ref{1} } ${}$\\
 %%%%%%%%%%%%%%%%%%%%%%%%%%%%%%%%%%%%%%%%%%%%%%%%%%%%%%%%%%%%%%%%%%%%%%%%%%%
 %%%%%%%%%%%%%%%%%%%%%%%%%%%%%%%%%%%%%%%%%%%%%%%%%%%%%%%%%%%%%%%%%%%%%%%%%%%
 Since a configuration $u\in {\cal C}_{\cal P}^{L}$ is a curve $u:[0,L]\ap \mathbb{S}_{\mathbb{H}}$, we can naturally define an  action  of $\mathfrak{M}_{HS}(\Sp_{\mathbb{H}})$ on $ {\cal C}_{\cal P}^{L}$  (again denote by $\mathfrak{A}$) by
 $$\mathfrak{A}(\phi,u)(s)=\phi(u(s))\textrm{ for } s\in [0,L].$$
 Since  the action of  $\mathfrak{M}_{HS}(\Sp_{\mathbb{H}})$  on $\mathbb{S}_{\mathbb{H}}$ is smooth and effective,  the same is true for the action on $ {\cal C}_{\cal P}^{L}$.

 Let  $\mathfrak{a}:\mathfrak{m}_{HS}(\Sp_{\mathbb{H}})\ap \textrm{Vect}({\cal C}_{\cal P}^{L})$ be the associated infinitesimal action where  $\textrm{Vect}({\cal C}_{\cal P}^{L})$
denote the space of vector fields on ${\cal C}_{\cal P}^{L}$.  As previously, we identify $\mathfrak{m}_{HS}(\Sp_{\mathbb{H}})$ with $\mathfrak{g}_\infty$, and we have (cf \cite{Ha} or \cite{Ro})
$$\mathfrak{a}([U_i,U_j])=-[\mathfrak{a}(U_i),\mathfrak{a}(U_j)].$$

 Moreover, according to Proposition \ref{action S} and the characterization  (\ref{expnabla}) of $\textrm{grad}\phi$ and the definition of $E_{i}$, we have
 \begin{eqnarray}\label{infactionC}
\mathfrak{a}(U_{i})=E_{i} \textrm{ and } \mathfrak{a}([U_i,U_j])=\mathfrak{a}(\O_{ij})=-[E_{i},E_{j}].
\end{eqnarray}

Of course, we also have a bundle morphism (again denoted $\mathfrak{a}$):
$$ \mathfrak{a}: \mathfrak{g}\times {\cal C}_{\cal P}^{L}\ap T{\cal C}_{\cal P}^{L}.$$

Now, we consider the restriction $\mathfrak{A}^1$ of the previous action $\mathfrak{A}$ to $\mathfrak{M}^1_{HS}(\Sp_{\mathbb{H}})$ on $ {\cal C}_{\cal P}^{L}$ and we also have the same relation (\ref{infactionC}) for the restriction $\mathfrak{a}^1$ of $\mathfrak{a}$ to the Lie algebra $\mathfrak{g}_1=\mathfrak{m}^1_{HS}(\Sp_{\mathbb{H}})$ of $\mathfrak{M}^1_{HS}(\Sp_{\mathbb{H}})$.

  According to the notations of  Section 4.3 of  \cite{PeSa} the Banach space  $\mathbb{G}^2$ is isomorphic to   $\mathfrak{g}$. Therefore we have $\mathfrak{a}(\mathfrak{g}\times\{u\})={\cal D}(u)$. Therefore, from the proof of Lemma 4.4 and Claim 1, we obtain that  the orbit of the action $\mathfrak{A}$ through $u$ is exactly the maximal integral manifold of $\cal D$ through $u\in {\cal C}_{\cal P}^{L}$.\\

  On the other hand the orbit ${\cal O}^1(u)$ of the action $\mathfrak{A}^1$ through $u\in{\cal C}_{\cal P}^{L}$ is contained in the orbit ${\cal O}(u)$ of $\mathfrak{A}$ through $u$. Moreover  ${\cal O}^1(u)$ is dense in ${\cal  O}(u)$. But according to Theorem \ref{subR} we can obtain the inclusion ${\cal O}^1(u)\subset{\cal A}(u)$. Therefore  the proof of Theorem \ref{1} is complete.

%%%%%%%%%%%%%%%%%%%%%%%%%%%%%%%%%%%%%%%%%%%%%%%%%%%%%%%%%%%%%%%%%
\section{Appendix}\label{A}
%%%%%%%%%%%%%%%%%%%%%%%%%%%%%%%%%%%%%%%%%%%%%%%%%%%%%%%%%%%%%%%%%
%%%%%%%%%%%%%%%%%%%%%%%%%%%%%%%%%%%%%%%%%%%%%%%%%%%%%%%%%%%%%%%%%
\subsection{Appendix A1: proof of  Theorem \ref{prodexp} }\label{A1}${}$\\
%%%%%%%%%%%%%%%%%%%%%%%%%%%%%%%%%%%%%%%%%%%%%%%%%%%%%%%%%%%%%%%%
%%%%%%%%%%%%%%%%%%%%%%%%%%%%%%%%%%%%%%%%%%%%%%%%%%%%%%%%%%%%%%%
Given a Hilbert space $\mathbb{H}$,  we denote by $SO_{HS}(\mathbb{H})$ the  Hilbert-Schmidt Lie group $SO(\mathbb{H})\cap GL_{HS}(\mathbb{H})$  provided with the topology of   the Hilbert-Schmidt norm and by $\mathfrak{so}_{HS}(\mathbb{H})$ its Lie algebra. At first we prove the following result (cf \cite{GaXu} for finite dimension)

\begin{Pro}\label{P}${}$\\
 The map  $\mathrm{Exp}:\mathfrak{so}_{HS}(\mathbb{H})\ap SO_{HS}(\mathbb{H})$ is surjective. More, precisely for each $Q\in SO_{HS}(\mathbb{H})$, there exists a family $\{\theta_j\}_{j\in J}$ with $0<\theta_j\leq \pi$ and a family of $\{B_j\}_{j\in J}$ with $B_j\in \mathfrak{so}_{HS}(\mathbb{H})$ such that $[B_k,B_j]=0$ for $k\not=j$ and $(B_j)^3=-B_j$  so that
$$Q=\dis\prod_{j\in J}\mathrm{Exp}(\theta_jB_j)=\mathrm{Exp}(\sum_{j\in J}\theta_j B_j).$$
Moreover, if $n_j$ is the rank of $B_j$ then $(|B|_{HS})^2=\dis\sum_{j\in J}n_j(\theta_j)^2$, where $ B=(\sum_{j\in J}\theta_j B_j)$.
\end{Pro}

\begin{proof}

Let  $B\in \mathfrak{so}_{HS}(\mathbb{H})$, $B$  is a compact operator skew-adjoint. Therefore, in the complexification  $\mathbb{H}^{C}$  of $\mathbb{H}$, we can write $B=iA$, where $A$ is a self adjoint compact operator. It follows that  the eigenvalues of $B$ are of type $\{\pm i\l_j\}_{j\in J}$ where $J$ is a finite or countable set and  $\{\l_j\}$ is a  strictly positive decreasing sequence which converges to $0$ if $J$ is countable. From classical spectral theory we have:
\begin{eqnarray}\label{spectdecomp}
\mathbb{H}=\dis\bigoplus_{j\in J} \mathbb{E}_j\oplus \mathbb{K}
\end{eqnarray}
where  $\mathbb{E}_j$ is the subspace such that the restriction of $B$ to  $\mathbb{E}_j$ is  $\pm i\l_j Id_{\mathbb{E}_j}$ and $\mathbb{K}$ is the kernel of $B$. Moreover, each $\mathbb{E}_j$ is orthogonal to $\mathbb{E}_k$ and $\mathbb{K}$ for $k\not=j$.  In particular, $\mathbb{E}_j$ is an even finite dimensional space. We can choose a Hilbert  basis $\cup_{j\in J}\{e_{l_1},\cdots, e_{2l_j}\}\cup \{ e_l, l\in L\}$ of $\mathbb{H}$ such that $\{e_{l_1},\cdots, e_{2l_j}\}$ is a basis of  $\mathbb{E}_j$ and $\{e_l, l\in L\}$ is a basis of $\mathbb{K}$. Moreover such a choice can be done such that the restriction of $B$ to $\mathbb{E}_j$ is of type $\l_j\bar{ B}_j$  where $\bar{B}_j$ has a matrix of the form

\begin{eqnarray}\label{decompoB}
\begin{pmatrix} J_{l_1}&\cdots&0&\cdots&0\\
				0&\cdots&J_{l_r}&\cdots&0\\
				\cdots&\cdots&\cdots&\cdots\\
				0&\cdots&0&\cdots&J_{l_j}\\
				\end{pmatrix}
\end{eqnarray}

where each block $J_{l_r}=\begin{pmatrix} 0&-1\\
								1&0\\
								\end{pmatrix}$. From this construction, we see that $\{\pm i\l_j\}_{j\in j}$ is the set of  non zero eigenvalues of $B$ and $\mathbb{E}_j$ is the eigenspace associated to $\pm i\l_j$. \\
								
								 Let  $B_j$ be the endomorphism whose  restriction to $\mathbb{E}_j$  is $\dis\frac{1}{\l_j}B_{| \mathbb{E}_j}$ and which is $0$ on $(\mathbb{E}_j)^\perp$. By construction, we have:
$$B=\dis\sum_{j\in J}\l_jB_j,\;\;\;\; [B_k,B_j]=0, \textrm{ for } k\not=j, \textrm{ and } (B_j)^3=-B_j.$$
It follows that we get
\begin{eqnarray}\label{proddecompo}
Q=\textrm{Exp}B=\textrm{Exp}(\dis\sum_{j\in J}\l_jB_j)=\dis\prod_{j\in J}\textrm{Exp}(\l_jB_j).
\end{eqnarray}
In particular, the eigenvalues of $Q$ which are different from $1$ is the family $e^{\pm i \l_j}$. Thus  in (\ref{proddecompo})  each $e^{\pm i \l_j}$ can be written $e^{\pm i\theta_j}$ with $0<\theta_j\leq \pi$. and we have
		$$(|B|_{HS})^2=2\dis\sum_{j\in J}n_j(\theta_j)^2$$
where $n_j= $dim$\mathbb{E}_j$.\\

Conversely, consider any $Q\in SO_{HS}(\mathbb{H})$. Then, $C=Q-Id$ is compact and so the set of eigenvalues of $Q$ different from $1$ is at most countable. Since $Q$ is unitary  of  a real Hilbert space , we can write this set as $\{e^{\pm i\theta_j}\}_{j\in J}$. Note that each eigenspace of $Q$ is an eigenspace of $C$ and conversely. Moreover,  the  set of non zero eigenvalues of $C$ is  $\{e^{\pm i\theta_j}-1\}_{j\in J}$. Therefore we have a spectral decomposition associated to $C$ of type (\ref{spectdecomp}) where $\mathbb{K}$ is the kernel of $C$. Note that  the  restriction $Q_j$  of $Q$ to each finite dimensional space $\mathbb{E}_j$ is an isometry of this space whose eigenvalues are $\{e^{\pm i\theta_j}\}$. According to the classical Lemma of decomposition of rotations in finite dimension, (see \cite{Au} for instance),  we have an orthogonal basis $\{e_{l_1},\cdots, e_{2l_j}\}$ of $\mathbb{E}_j$  in which  $Q_j$ has a matrix of the form:
$$\begin{pmatrix} R_{l_1}&\cdots&0&\cdots&0\\
				0&\cdots&R_{l_r}&\cdots&0\\
				\cdots&\cdots&\cdots&\cdots\\
				0&\cdots&0&\cdots&R_{l_j}\\
				\end{pmatrix}$$

where each block $R_{l_r}=\begin{pmatrix} \cos \theta_{l_r}&-\sin \theta_{l_r}\\
								\sin \theta_{l_r}&\cos \theta_{l_r}\\
								\end{pmatrix}$. In fact we must have
								$$\theta_{l_r}\equiv\theta_j (\textrm{ modulo }\pi).$$
It follows that  we have $Q_j=\mathrm{Exp}(\theta_j\bar{B}_j)$ where $\bar{B}_j$ has a matrix of type (\ref{decompoB}) in the previous basis. As in the first part, let  $B_j$ be the endomorphism which is equal to $\bar{B}_j$ on $\mathbb{E}_j$ and is zero on $(\mathbb{E}_j)^\perp$.\\
On the other hand, let  $\hat{Q}_j$ be  the  invertible operator whose restriction to $\mathbb{E}_j$ is equal to $Q_j$ and which is the identity on $(\mathbb{E}_j)^\perp$. Of course the infinite composition $\dis\prod_{j\in J}\hat{Q}_j$ is equal to $Q$ and we get
$$Q=\dis\prod_{j\in J}\mathrm{Exp}(\theta_j{B}_j).$$
As in the first part, by construction, we again have $[B_k,B_j]=0$ it follows that $B=\dis\sum_{j\in J}\theta_j B_j$ is well defined and $|B|^{2}_{HS}=2\dis\sum_{j\in J}n_{j}(\theta_j)^2$.\\
\end{proof}

We also need the following result (see \cite{Ga} for finite dimension).

\begin{Pro}\label{T}${}$\\
Given a boost $T\in SO_{HS}(\mathbb{H},1)$, there exists $ U\in \mathfrak{h}$ such that $T=\mathrm{Exp}(U)$.
\end{Pro}

The proof of this Proposition is a formal  adaptation of the corresponding result  in finite dimension of \cite{Ga}. We only give the essential arguments.

\begin{proof}

Let $U\in \mathfrak{h}$. We have  $U=\begin{pmatrix} 0& [u]^*\\
							[u]&0\\
							\end{pmatrix}$
	where $u\in \mathbb{H}$. We have $U^3= \o^2U$ where $\o=|u|$. By application of this relation  we easily get
	
$$\mathrm{Exp}(U)=Id_{\mathcal{H}}+\dis\frac{\sinh \o}{\o} U+\dis\frac{\cosh \o-1}{\o^2} U^2.$$
	As in finite dimension we obtain:
	$$\mathrm{Exp}(U)=\begin{pmatrix} \cosh \o&\dis\frac{\sinh \o}{\o}[u]^*\\
								\dis\frac{\sinh \o}{\o}[u]&Id_{\mathbb{H}}+\dis\frac{\cosh \o-1}{\o^2}[u][u]^*\\
								\end{pmatrix}. $$
We have the relation
$$\Big{(}Id_{\mathbb{H}}+\dis\frac{\cosh \o-1}{\o^2}[u][u]^*\Big{)}^2=Id_{\mathbb{H}}+\dis\frac{\sinh^2 \o}{\o^2}[u][u]^*$$
Finally, we get
$$\mathrm{Exp}(U)=\begin{pmatrix} \cosh \o&\dis\frac{\sinh \o}{\o}[u]^*\\
								\dis\frac{\sinh \o}{\o}[u]&\sqrt{Id_{\mathbb{H}}+\dis\frac{\sinh^2 \o}{\o^2}[u][u]^*}\\
								\end{pmatrix} .$$

On the other hand, from the proof of Proposition \ref{decompA}
we have $T=\begin{pmatrix}c&[v]^*\\
			[v]&\sqrt{Id_\mathbb{H}+[v ].[v]^*}\\
			\end{pmatrix}$ for some $v\in \mathbb{H}$.\\
Given $v\in \mathbb{H}$ we have then  to find $u\in \mathbb{H}$ which satisfies the following   equation:
$$\begin{pmatrix}c&[v]^*\\
			[v]&\sqrt{Id_{\mathbb{H}}+[v ].[v]^*}\\
			\end{pmatrix}=\begin{pmatrix} \cosh \o&\dis\frac{\sinh \o}{\o}[u]^*\\
								\dis\frac{\sinh \o}{\o}[u]&\sqrt{Id_{\mathbb{H}}+\dis\frac{\sinh^2 \o}{\o^2}[u][u]^*}\\
								\end{pmatrix} .$$
								
This equation can be solved as in finite dimension, point by point (cf \cite{Ga}).

\end{proof}

\begin{proof}[Proof of Theorem \ref{prodexp}]${}$\\ Let  $A\in SO_{HS}(\mathbb{H},1)$. From   Proposition \ref{CMMob}, we have  $A=PT$ where  $P=\begin{pmatrix}1&0\\
0&Q\\
\end{pmatrix} $ and  $Q$ belongs to $SO(\mathbb{H})$ and where $T$ is a boost.\\

From Proposition \ref{T}, there exists a family of endomorphisms $\{B_j\}_{j\in J}$ with $B_j\in\mathfrak{so}_{HS}(\mathbb{H})$ such that\\ $[B_k,B_j]=0$ for $j\not=k$  and a sequence $\{\theta_j\}_{i\in J}$ with $0<\theta_j\leq\pi$ so that $Q=\dis\prod_{j\in J}\mathrm{Exp}(\theta_j{B}_j)$. According to the isomorphism $Q\ap P=\begin{pmatrix}1&0\\
0&Q\\
\end{pmatrix} $ from $\mathfrak{so}_{HS}(\mathbb{H})$ to $\mathfrak{s}$, we may assume that $B_j$ belongs to $\mathfrak{s}$. It follows that we get
$$P=\dis\prod_{j\in J}\mathrm{Exp}(\theta_j{B}_j).$$
 According to Proposition \ref{T} the proof is complete.\\					
		\end{proof}

%%%%%%%%%%%%%%%%%%%%%%%%%%%%%%%%%%%%%%%%%%%%%%%%%%%%%%%%%%%%%%%%%%%%%%%%%%%%%%%%%
\subsection{Appendix A2:  proof of Lemma \ref{bong}} \label{A2}${}$\\
%%%%%%%%%%%%%%%%%%%%%%%%%%%%%%%%%%%%%%%%%%%%%%%%%%%%%%%%%%%%%%%%%%%%%%%%%%%%%%%%%
%%%%%%%%%%%%%%%%%%%%%%%%%%%%%%%%%%%%%%%%%%%%%%%%%%%%%%%%%%%%%%%%%%%%%%%%%%%%%%%%%
We first recall some result about sub-Riemannian geometry on $SU(1,1)$. At first, we can identify $\R^2$ with the complex space $\C$ it is classical that $SO_0(2,1)$ is isomorphic to $PSU(1,1)$ which is the connected components of the identity of the Lie group $SU(1,1)$. It follows that  $SU(1,1)$ is the group of invertible matrices of type
$\begin{pmatrix} z_1& z_2\\
			\bar{z}_1&\bar{z}_2\\
			\end{pmatrix}$
			where $z_1$ and $z_2$ belongs to $\C$.   Note that $SU(1,1)$ can be identified with $\C \times \mathbb{S}^1$. The Lie algebra $\mathfrak{su}(1,1)$ of $SU(1,1)$ is generated by:
	\begin{center}	
		$X=\dis\frac{1}{2}\begin{pmatrix} 0&-1\\
						-1&0\\
						\end{pmatrix}\;\;\;\;$   $Y=\dis\frac{1}{2}\begin{pmatrix} 0&i\\
						-i&0\\
						\end{pmatrix}\;\;\;\;$ $Z=\dis\frac{1}{2}\begin{pmatrix} -i&0\\
						0&i\\
						\end{pmatrix}$
						\end{center}
						We have the bracket relations
						\begin{center}
$[X,Y]=-Z,\;\;\;\;\;$ $[X,Z]=-Y,\;\;\;\;\;$ $[Y,Z]=X,\;\;\;\;\;$
\end{center}
On $SU(1,1)$ we consider the left invariant distribution $\D$ generated by $X$ and $Y$ and the left  invariant Riemannian metric induced by $\dis\frac{1}{2}Tr(X_1X_2)$ on the subspace generated by $X$ and $Y$. We get a sub-Riemannian structure $(SU(1,1),\D,g)$ on $SU(1,1)$. Let $\d$ be the left-invariant horizontal distance associate to this structure. The universal covering $\tilde{SU}(1,1)$ can be identified with $\C\times\R$. The canonical projection $\r: \tilde{SU}(1,1)\ap {SU}(1,1)$ is given by:
$$(z,t)\ap \begin{pmatrix}\sqrt{1+|z|^2}e^{it}&z\\
			\bar{z}&\sqrt{1+|z|^2}e^{-it}\\
			\end{pmatrix}$$
 %We can lift the sub-Riemannian structure on ${SU}(1,1)$, on a sub-Riemannian structure $(\tilde{SU}(1,1),\tilde{\D},\tide{g})$ on $\tilde{SU}(1,1)$. We denote by $\tilde{\d}$ the associated horizontal distance
For our purpose, we need only the following partial result of \cite{GrVa}:

\begin{Pro} \label{SU11} ${}$\\
Let $A= \begin{pmatrix} e^{it}&0\\
			0&e^{-it}\\
			\end{pmatrix}$ with $ t\not=0$. There exists a (normal) minimal length horizontal geodesic   which joins $Id$ to $A$ and  the horizontal distance $\d(Id,A)=|\theta|$ for $0< |\theta|\leq \pi$ and $\pm \theta\equiv t\; (mod \;\pi)$.
			\end{Pro}
Now,  we have an isomorphism from $SU(1,1)$ to  $SO(2,1)$  given by:

$$\begin{pmatrix}\sqrt{1+|z|^2}e^{it}&z\\
			\bar{z}&\sqrt{1+|z|^2}e^{-it}\\
			\end{pmatrix}	\ap	\begin{pmatrix}1&0&0\\
			0&e^{it}&0\\
			0&0&e^{-it}\\
			\end{pmatrix}\begin{pmatrix}\sqrt{1+|z|^2}&\textrm{Re}(z)&\textrm{Im}(z)\\
			\textrm{Re}(z)&a&b\\
			\textrm{Im}(z)&b&c\\
			\end{pmatrix}$$
where ${\begin{pmatrix}a&b\\
			b&c\\
			 \end{pmatrix}}^2=\begin{pmatrix}\textrm{Re}(z)^2&\textrm{Re}(z)\textrm{Im}(z)\\
			\textrm{Re}(z)\textrm{Im}(z)&\textrm{Im}(z)^2\\
			\end{pmatrix}$.\\
The induced isomorphism between Lie algebra  is then
$\begin{pmatrix}it&z\\
\bar{z}&-it\\
\end{pmatrix}\ap \begin{pmatrix}1& \textrm{Re}(z)&\textrm{Im}(z)\\
					\textrm{Re}(z)& \cos t&-\sin t\\
					\textrm{Im}(z)&\sin t&\cos t\\
					\end{pmatrix}$.
As a consequence we get  an isomorphism between the sub-Riemannian structure on $SU(1,1)$ and the sub-Riemannian structure on $SO(2,1)$.

\begin{proof}[Proof of Lemma \ref{bong}]${}$\\
Recall that  $\bar{A}_j=\begin{pmatrix}1&0\\
							0&\mathrm{Exp}(\bar{\cal B}_j)\\
							\end{pmatrix}$
and $\bar{\cal B}_j$ has a decomposition (\ref{decompoB})   in  diagonal blocks $ \theta_jJ_{l_r}$. Each block $J_{l_r}$ gives rise to an element of $SO(\mathbb{F}_{l_r},1)$ where $\mathbb{F}_{l_r}$ is a plane in $\mathbb{E}_j$. Therefore, according to Proposition \ref{SU11} , via the previous isomorphism, we have a horizontal curve in $\bar{\g}_{l_r}:[0,T_{l_r}]\ap SO(\mathbb{F}_{l_r},1)$ arc-length parameterized whose length is $\theta_j$ such that $\bar{\g}_{l_r}(0)=Id_{\mathbb{F}_{l_r}}$ and $\bar{\g}_{l_r}(T_{l_r})= \theta_jJ_{l_r}$. In particular $T_{l_r}=|\theta_j|$. We get a curve $\g_{l_r}: [0,\theta_j]\ap SO(\mathbb{E}_j,1)$ of length $\theta_j$, which joins $Id_{\mathbb{E}_j}$ to some element $\theta_j\hat{J}_{l_r}$ of $ SO(\mathbb{E}_j,1)$
$$(\g_{l_r})_{| \mathbb{F}_{l_r}}=\bar{\g}_{l_r},\;\textrm{ andÊ} (\g_{l_r})_{| [\mathbb{F}_{l_r}]^\perp}=id.$$
It follows that  the curve
$\g_j$, obtained by concatenation of the family $\g_{l_r}$ for $ l_r=1,\cdots l_j$, is  defined on $[0,n_j\theta_j]$, $\g_j$ is an horizontal curve in  $ SO(\mathbb{E}_j,1)$, of length $n_j\theta_j$ which  joins $Id_{\mathbb{E}_j}$ to $A_j$.

\end{proof}

\end{document}